\documentclass[sn-mathphys]{sn-jnl}

\jyear{2021}%
\usepackage{color}
\usepackage{float}
\theoremstyle{thmstyleone}%
\newtheorem{theorem}{Theorem}
\newtheorem{proposition}[theorem]{Proposition}%

\theoremstyle{thmstyletwo}%
\newtheorem{lemma}[theorem]{Lemma}

\theoremstyle{thmstylethree}%

\begin{document}

\title[Numerical analysis of a time discretized method for nonlinear filtering problem]{Numerical analysis of a time discretized method for nonlinear filtering problem with L\'evy process observations}%

\author[1]{\fnm{Fengshan} \sur{Zhang}}\email{858012011@qq.com}

\author[1]{\fnm{Yongkui} \sur{Zou}}\email{zouyk@jlu.edu.cn}

\author[1]{\fnm{Shimin} \sur{Chai}}\email{chaism@jlu.edu.cn}

\author*[2]{\fnm{ Yanzhao} \sur{Cao}}\email{yzc0009@auburn.edu}

\affil[1]{\orgdiv{School of Mathematics}, \orgname{Jilin University}, \orgaddress{\street{Qianjin Street}, \city{Changchun}, \postcode{130012}, \state{Jilin}, \country{China}}}

\affil*[2]{\orgdiv{Department of Mathematics and Statistics}, \orgname{Auburn University}, \orgaddress{\street{AL}, \city{Auburn}, \postcode{36840}, \state{Alabama}, \country{America}}}


\abstract{In this paper,  we consider   a nonlinear filtering model with observations driven by correlated Wiener processes and point processes. We first derive a Zakai equation whose solution is an unnormalized   probability density function of the filter solution.  Then we apply a splitting-up technique to decompose the Zakai equation into three  stochastic differential equations, based on which we construct a splitting-up approximate solution and prove  its half-order convergence. Furthermore,  we   apply a finite  difference method to  construct  a time semi-discrete   approximate solution to the splitting-up system and prove its half-order convergence to the exact solution of the Zakai equation.  Finally, we present some numerical experiments to demonstrate the theoretical analysis.}

\keywords{nonlinear filter, splitting-up technique, difference method, convergence order}



\maketitle

\section{Introduction}
\label{sec1}
The aim of a nonlinear filtering problem is to seek the conditional expectation, which is   the best estimate of the unobserved state of a stochastic dynamical system given its partial observation. The
observation  is usually  described as  a nonlinear stochastic differential equation driven by a noise process. In many applications,  such as  biology \cite{dawson2015}, physics \cite{MiSK1979},
target tracking \cite{yang2012} and weather forecast  \cite{AnML2004}, the noise can be characterized by a  standard Wiener process.   However, in some applications, such as  the number of customers arriving at a supermarket \cite{DaSu2013} and  the number of births in a given period of time \cite{Young1998}, the noise is governed by 
a point process.   In some other applications such as the  credit risk models  \cite{FrSch2012,FRXL2013},   mathematical finance \cite{bremaud1972,FrRS2009} and insurance \cite{AgLa2007, CeCl2015},  the noise  can be described by a mixture of  a Wiener process and a point process, which is usually called L\'evy process.

There have been a few theoretical and numerical studies on  nonlinear filtering problems  driven  by {\color{blue}L\'evy} processes.  Qiao and Duan \cite{QiDu2015} studied a nonlinear filtering  model where both the state and observation involve point processes. They simultaneously derived the  Zakai equations and Kushner-Stratonovich equations  and  proved their  well-posedness.  
Fernando and Hausenblas \cite{FeHa2018} investigated a  nonlinear filter model with correlated point processes for the state and observation. They provided  sufficient conditions for the
well-posedness of the corresponding Zakai equation. Frey etc.\ \cite{FRXL2013}  used the PDF filter method to approximate
a nonlinear filter model driven by point processes and independent Wiener processes.  The PDF filter method is designed to directly approach the conditional density function, which  satisfies a stochastic partial differential equation, namely  Zakai equation  \cite{zakai1969}. In  \cite{FRXL2013} the authors applied a spectral Galerkin method to  set up a spatial semi-discrete equation and proved that its solution  converges to the exact solution of the Zakai equation. However, they did not provide the convergence order.  They also used the Euler-Maruyama scheme and a splitting-up method, to discretize temporal variables. 

In this paper, we use the splitting-up method to investigate the numerical approximation of  a nonlinear filtering model where the observation is driven  by  the mixture of point processes and correlated Wiener processes.  The splitting-up method \cite{BaGrRa,FlFr1991,GyIs2003} is a well-known strategy for solving
Zakai equations. It   decomposes the Zakai equation into a system consisting of deterministic PDEs and stochastic differential equations (SDEs) \cite{Ito1996,AnML2004,BaoCao2014,BaGrRa,FlFr1991}. 
Our contribution in this paper is twofold.  First, we  decompose the Zakai equation into three equations: an SDE driven by the Wiener process, a second order parabolic equation satisfying the uniform elliptic 
condition,  and an SDE driven by a  point process. Through  the solution operators of the three equations and their a priori estimates, we construct a splitting-up approximation and  prove that it converges to the Zakai solution with first  order accuracy.    We note that in some references \cite{FPLe1991,LeFr1992,GyIs2003} concerning the nonlinear filtering models with correlated noises, the decomposed second-order parabolic equation is possibly degenerate, which may cause difficulty for numerical implementations. Our second contribution is the derivation of the half-order convergence of the time semi-discrete approximation. To the best of our knowledge, this is the first time a convergence order of a numerical method for nonlinear filtering problems with jump processes has been provided.  

This paper is organized as follows.
In section \ref{sec2}, we introduce a nonlinear filtering model with the mixed noise of point process and correlated Wiener process and then derive the corresponding Zakai equation. In section \ref{sec3}, we apply a  splitting-up method  to construct a  splitting-up approximate solution to the Zakai equation and  establish   a priori
estimates for the splitting-up solution and show that the convergence  is of half order.  In section \ref{sec4}, we use  finite difference methods  to construct a time semi-discrete approximation
and prove that the semi-discrete solution converges to the exact solution with half order.  Finally in  section \ref{sec5}, we present some numerical experiments to illustrate our theoretical analysis.

\section{A nonlinear filtering model with  jump observations and its Zakai equation}\label{sec2}
In this section, we first introduce a nonlinear filtering model whose  observations are driven by  L\'evy processes. Then we derive the corresponding Zakai equation which characterizes the development of the density function of the filtering solution process. Finally, we investigate the regularity of the solution of  the Zakai equation.

\subsection{A nonlinear filtering model}
In this subsection, we introduce  a nonlinear filtering model with noises simultaneously driven by a point and correlated Wiener processes and then  discuss some basic assumptions.

Let $(\Omega,{\cal{F}},P)$ be a given probability space.
Consider a nonlinear filtering model whose  state (or signal) process $X_t$ and two observation processes $Y_t$ and $Z_t$ are given by 
\begin{align}
	\label{eq:2.1.1}
	&X_{t} = X_{0} + \int_{0}^{t}g(X_{s},Y_s)ds + \int_{0}^{t}\sigma(X_{s})dw_{s},\quad 0\leq t\leq T,
	\\
	&  \label{eq:2.1.2}
	Y_t = Y_0+\int_0^th(X_s)ds + \int_0^tb(Y_s)dw_s + \int_0^t\tilde{b}(Y_s)dv_s,\quad 0\leq t\leq T,
	\\
	\label{eq:2.1.2-0}
	& Z_t ~\text{is a doubly stochastic Poisson process with density function}~ \lambda(X_t),
\end{align}
where $w_t\in\mathbb{R}^{m_1}$ and $v_t\in \mathbb{R}^{m_2}$ are two  standard independent  Winner processes, 
$Z_t$ is  a doubly stochastic Poisson process with a continuous density function $\lambda: \mathbb{R}^d\to [\varpi_1, \varpi_2]\subset
\mathbb{R}_+$ such that $Z_t-\int_0^t\lambda(X_s)ds$ is a martingale. The corresponding jump times for $Z_t$  are random variables denoted by $\tau_1<\tau_2<\cdots<\tau_{n_0}$, where $n_0$ is an integer-valued random variable.

The objective of the nonlinear filtering problem is to seek an optimal estimation of   $X_t$ based on observations $Y_t$ and $Z_t$, which is characterized by  the  conditional
expectation $E[X_t\vert Y_t,Z_t]$.
Now, we describe in detail the assumptions used in this work.

\begin{itemize}
	\item[\bf{H1}]
	$E\vert X_0\vert^2+E\vert Y_0\vert^2<\infty$.
	\item[\bf{H2}]  $g:\mathbb{R}^d\times\mathbb{R}^q\to\mathbb{R}^d$ and   $h:\mathbb{R}^d\to\mathbb{R}^q$  are  bounded, continuous and square integrable,
	$\sigma:\mathbb{R}^d\to\mathbb{R}^{d\times m_1}$ is in  $C^2$ with bounded  first and second order derivatives, and $b$, $\tilde{b}$ are in $C^1$ with bounded first order derivatives.
	 
\end{itemize}

Define two families of symmetric non-negative matrix
\[
A(x)=(a_{i,j}(x))_{d\times d}=\frac{1}{2}\sigma(x)\sigma^T(x), \quad D(y)=b(y)b(y)^{T}+\tilde{b}(y)\tilde{b}(y)^{T}.
\]

\begin{itemize}
	\item[\bf{H3}]
	There exist two constants  $0<\alpha_1<\alpha_2$  such that for any $x\in \mathbb{R}^d$, $y\in\mathbb{R}^{q}$ and  $u\in\mathbb{R}^q$, there hold
	\begin{eqnarray*}
		&\alpha_1\|u\|^2\leq u^TA(x)u\leq \alpha_2 \|u\|^2,&
		\\
		&\alpha_1\|u\|^2 \leq u^T D(y)u\leq \alpha_2 \|u\|^2.&
	\end{eqnarray*}
\end{itemize}

Define two  Sobolev  spaces
$H=L^{2}(\mathbb{R}^{d})$ with norm $\|\cdot\|$ and $V=H^{1}_0(\mathbb{R}^{d})$ with norm $\|\cdot\|_1	$. 
By $V'$, we denote the dual space of $V$.  Obviously, for any  $\forall \phi\in{V}\subset H \subset V'$ there holds
\begin{equation}\label{eq:2.1.2-1}
	\|\phi\|_{V^\prime}\leq\|\phi\|\leq \|\phi\|_1.
\end{equation}

Define a filtration associated with the observations  by 
$$
\mathcal{F}_t=\sigma(Y_s,Z_s,~0\leq s\leq t),\quad  t\in [0,T],
$$ 
which is right continuous and complete. By  $\mathbb{L}^{2}(0,T;V)$ we denote  a  Hilbert space consisting  of $\mathcal{F}_{t}$ progressively measurable $V$-valued stochastic process $z(t)$ with $E\int_{0}^{T}\|z(t)\|_1^{2}dt<\infty$.

\subsection{Zakai equation and its  regularity}
The main task of this subsection is to derive the Zakai equation of the nonlinear filtering model \eqref{eq:2.1.1}-\eqref{eq:2.1.2-0} and study the regularity of its solution.

Assume that $X_t$ is the solution process of \eqref{eq:2.1.1}. For any $\phi\in C_0^{\infty}(\mathbb{R}^d)$,
 define $\pi_t(\phi)$ as the conditional expectation of $\phi(X_t)$ given  $\mathcal{F}_t$, i.e., \
\begin{equation}
	\label{eq:2.1.3}
	\pi_t(\phi):= E\big( \phi (X_t)\vert  \mathcal{F}_t \big).
\end{equation}
For any $t\in[0,T]$, define
\begin{equation}
	\label{eq:2.1.4}
	\begin{aligned}
		\eta_t={}&\prod\limits_{\tau_m\leq t}\lambda(X_{\tau_m-})\cdot\exp\left(\int^{t}_{0}h^T(X_s)D^{-1}(b(Y_s)dw_s+\tilde{b}(Y_s)dv_s)\right.\\
		&
		\left.+\frac{1}{2}\int^{t}_{0}h^T(X_s)D^{-1}h(X_s)ds-\int_0^t(\lambda(X_s)-1)ds \right).
	\end{aligned}
\end{equation}
According to  Novikov Criterion  \cite[Theorem 41]{Protter2005}, $\eta_t$ is a nonnegative martingale if {\bf H1}-{\bf H3} hold. Define  a
new probability measure $\tilde{P}$ by virtue of  the Radon-Nikodym derivative $\frac{d\tilde{P}}{dP}=\eta_t^{-1}$.
The Girsanov theorem \cite{BaCr2009} implies that  $\bar{Y}_t=\int_0^tD(Y_s)^{-1/2}dY_s$ is a standard Wiener process  and $Z_t$ is a Poisson process with
intensity equal to $1$  under the probability measure   $\tilde{P}$. Furthermore, in the probability space $(\Omega,\mathcal{F},\tilde P)$ the three stochastic processes  $X_t$, $Y_t$ and $Z_t$ are independent of each other, and  the compensated Poisson process $N_t:=Z_t-t $ is a martingale.

Define a stochastic process
\begin{equation*}
	\tilde{Y}_t=\int_0^t[I-b^TD^{-1}b]^{-1/2}(d\tilde{w}_t-b^TD^{-1}dY_t),
\end{equation*}
where $\tilde{w}_t=w_t+\int_0^t b^T(Y_s)D^{-1}(Y_s)h(X_s)ds$ is a standard Wiener process under $\tilde{P}$. From \cite{Pard1981}, we have
\begin{lemma}
	\label{lem:2.1.1}
	Assume  {\bf H1}-{\bf H3}. Then 
	$\tilde{Y}_t$ is a standard Wiener process  independent of $Y_t$ under $\tilde{P}$ and equation (\ref{eq:2.1.1}) is equivalent to
	\begin{equation*}
		X_t=X_0\!+\!\!\int_0^t\!\![g(X_s)\!-\!B_1(X_s,Y_s)h(X_s)]ds\!+\!\!\int_0^t\!\!B_1(X_s,Y_s)dY_s\!+\!\!\int_0^t\!\!B_2(X_s,Y_s)d\tilde{Y}_s,
	\end{equation*}
	where $B_1(x,y)=\sigma(x) b^{T}(y)D^{-1}(y)$ and $ B_2(x,y)=\sigma(x)(I-b^T(y)D^{-1}(y)b(y))^{1/2}$.
\end{lemma}

Denote by  $\tilde{E}$ the expectation under the probability measure  $\tilde{P}$. The next proposition plays an important role in  the forthcoming analysis.
\begin{proposition}\cite[Proposition 3.15]{BaCr2009}
	\label{pro:2.1.1}
	Assume that {\bf H1}-{\bf H3} hold and let $U$ be an $\mathcal{F}_t$-measurable and integrable random variable. Then we have
	\begin{equation}
		\label{eq:2.1.4-2}
		\tilde{E}[U\vert\mathcal{F}_t]=\tilde{E}[U\vert\mathcal{F}_T].
	\end{equation}
\end{proposition}

By the Kallianpur-Striebel formula \cite[Proposition 3.16]{BaCr2009}, we have
\begin{equation}
	\label{eq:2.1.5}
	\pi_{t}(\phi)=\frac{\tilde{E}(\phi(X_{t})\eta_t\vert \mathcal{F}_t)}{\tilde{E}(\eta_t\vert \mathcal{F}_t)}
	=\frac{(p(t),\phi)}{(p(t),1)},
\end{equation}
where $p(t)$ is the unnormalized conditional density function of $\tilde E(X_t\vert \mathcal{F}_t)$.

\begin{theorem}
	\label{thm:2.1.1}
	Assume that {\bf H1}-{\bf H3} hold and  let $\rho_t(\phi)=(p(t),\phi)$. Then  $\rho_t$ satisfies a functional equation:  for $\forall \phi\in C_0^{\infty}(\mathbb{R}^d)$ and $0\leq t\leq T$,
	\begin{equation}
		\label{eq:2.1.6}
		\rho_{t}(\phi)=\rho_{0}(\phi)+\int_{0}^{t}\rho_{s}(\mathcal{L}\phi)ds+\int_{0}^{t}\rho_{s}(\mathcal{B}\phi)dY_{s}+\int_0^t\rho_{s-}\big( \phi(\lambda-1) \big)d(Z_s-s),\quad \tilde{P}-a.s.,
	\end{equation}
	where
	\begin{eqnarray*}
		\mathcal{L}\phi &=& \sum\limits_{i,j=1}^da_{i,j}(x)\frac{\partial^2\phi}{\partial x_i\partial x_j}+\sum\limits_{i=1}^d g_i(x)\frac{\partial \phi}{\partial x_i},\\
		\mathcal{B}\phi &=&\phi h^TD^{-1}+\sum\limits_{i=1}^d \frac{\partial \phi}{\partial x_i} b_{1,i},
	\end{eqnarray*}
	and $b_{1,i}$ denotes the $i$-th row of matrix
	$ B_1(x,y)$.
\end{theorem}

\begin{proof}
We approximate $\eta_t$ with $\eta_t^{\varepsilon}=\frac{\eta_t}{1+\varepsilon\eta_t}$. By  It\^o formula
\begin{equation}
	\label{eq:2.1.7}
	d\eta_t=\eta_{t-}\big[ h(X_t)^TD^{-1}dY_t+(\lambda(X_{t-})-1)d(Z_t-t) \big].
\end{equation}
Using It\^o's formula for the jump process, we have
\begin{equation}
	\label{eq:2.1.8}
	\begin{aligned}
		\eta_t^{\varepsilon}={}&\eta_0^{\varepsilon}+\int_0^t \frac{\eta_s}{(1+\varepsilon\eta_s)^2}h(X_s)^TD^{-1}dY_s-\int_0^t
		\frac{\varepsilon\eta_t^2}{(1+\varepsilon\eta_t)^3}h(X_s)^TD^{-1}h(X_s)ds\\
		{}&-\int_0^t\frac{\eta_s}{(1+\varepsilon\eta_s)^2}(\lambda(X_s)-1)ds+\sum_{\tau_m\leq t}\Delta \frac{\eta_{\tau_m}}{1+\varepsilon\eta_{\tau_m}}.
	\end{aligned}
\end{equation}
Let $X_t$ satisfy \eqref{eq:2.1.1} and for any
$\phi\in C_0^{\infty}(\mathbb{R}^d)$
\begin{equation}
	\label{eq:2.1.9}
	\begin{aligned}
		d\phi(X_{t})
		={}&\mathcal{L}\phi dt-\sum\limits_{i=1}^{d}( B_1h)^{i}\frac{\partial\phi}{\partial x_{i}}dt+\nabla\phi^T  B_1dY_{t}+\nabla\phi^T B_2d\tilde{Y}_t.
	\end{aligned}
\end{equation}
Applying   the product rule for semi-martingales to (\ref{eq:2.1.7}) and (\ref{eq:2.1.9}), we obtain
\begin{equation}
	\label{eq:2.1.10}
	\begin{aligned}
		&\eta_t^{\varepsilon}\phi(X_t)
		\\
		={}&\eta_0^{\varepsilon}\phi(X_0)+\int_0^t\eta_s^{\varepsilon} \mathcal{L}\phi(X_s)ds-\int_0^t \frac{\varepsilon\eta_s^2}{(1+\varepsilon\eta_s)^2}\nabla\phi^T  B_1hds+\int_0^t\eta_s^{\varepsilon}\nabla\phi^T
		B_1dY_s
		\\
		{}&+\int_0^t\eta_s^{\varepsilon}\nabla\phi^T B_2d\tilde{Y}_s+\int_0^t \frac{\eta_s\phi(X_s)}{(1+\varepsilon\eta_s)^2}h^TD^{-1}dY_s-\int_0^t \frac{\varepsilon\eta_s^2\phi(X_s)}{(1+\varepsilon\eta_s)^2}h^TD^{-1}hds
		\\
		{}&-\int_0^t \frac{\eta_s\phi(X_s)}{(1+\varepsilon\eta_s)^2}(\lambda(X_s)-1)ds+\int_0^t
		\frac{\eta_{s-}(\lambda(X_{s-})-1)\phi(X_{s-})}{(1+\varepsilon\eta_{s-}\lambda(X_{s-}))(1+\varepsilon\eta_{s-})}dZ_s.
	\end{aligned}
\end{equation}
According to Proposition \ref{pro:2.1.1}, we only need to compute the conditional expectation based on the filtration $\mathcal{F}_T$. Take conditional expectation about $\eta_t^{\varepsilon}\phi(X_t)$ based on the observation $\mathcal{F}_T$, then we have
\begin{equation}
	\begin{aligned}
		\label{eq:2.1.11}
		&\tilde{E}(\eta_t^{\varepsilon}\phi(X_t)\rvert\mathcal{F}_T)
		\\
		={}&\tilde{E}(\eta_0^{\varepsilon}\phi(X_0)\rvert\mathcal{F}_T)+\tilde{E}\bigg(\int_0^t\eta_s^{\varepsilon}\mathcal{L}\phi(X_s)ds\rvert\mathcal{F}_T
		\bigg)\\
		&-\tilde{E}\bigg( \int_0^t \frac{\varepsilon\eta_s^2}{(1+\varepsilon\eta_s)^2}\nabla\phi^T  B_1hds\rvert\mathcal{F}_T \bigg)+\tilde{E}\bigg(\int_0^t\eta_s^{\varepsilon}\nabla\phi^T
		B_1dY_s\rvert\mathcal{F}_T\bigg)\\
		&+\tilde{E}\bigg( \int_0^t
		\frac{\eta_s\phi}{(1+\varepsilon\eta_s)^2}h^TD^{-1}dY_s\rvert \mathcal{F}_T\bigg)+\tilde{E}\bigg(\int_0^t\eta_s^{\varepsilon}\nabla\phi^T B_2d\tilde{Y}_s\rvert\mathcal{F}_T\bigg)\\
		&-\tilde{E}\bigg( \int_0^t\frac{\varepsilon\eta_s^2\phi}{(1+\varepsilon\eta_s)^3}h^TD^{-1}hds\rvert\mathcal{F}_T\bigg)-\tilde{E}\bigg( \int_0^t \frac{\eta_s\phi}{(1+\varepsilon\eta_s)^2}(\lambda(X_s)-1)ds\rvert\mathcal{F}_T \bigg)\\
		&+\tilde{E}\bigg(\int_0^t\frac{\eta_{s-}(\lambda(X_{s-})-1)\phi(X_{s-})}{(1+\varepsilon\eta_{s-}\lambda(X_{s-}))(1+\varepsilon\eta_{s-})}dZ_s\rvert\mathcal{F}_T \bigg)\\
		:={}&E_1+E_2-E_3+E_4+E_5+E_6-E_7-E_8+E_9.
	\end{aligned}
\end{equation}

Now, we  show that as $\varepsilon\rightarrow 0$, the following limits hold in the sense $\tilde{P}$-a.s.,
\begin{equation*}
	\begin{aligned}
		&\tilde{E}(\eta_t^{\varepsilon}\rvert\mathcal{F}_T)\rightarrow\rho_t(\phi),\quad
		E_1\rightarrow\rho_0(\phi),\quad
		E_2\rightarrow\int_0^t\rho_s(\mathcal{L}\phi)ds,\quad
		E_3\rightarrow 0,\\
		&E_4\rightarrow\int_0^t\rho_s(\nabla\phi^T c)dY_s,\quad E_5\rightarrow\int_0^t\rho_s(h^TD^{-1}\phi)dY_s,\quad
		E_6\rightarrow 0, \quad  E_7\rightarrow 0,\\
		&E_8\rightarrow\int_0^t\rho_s(\phi(\lambda-1))ds,\quad
		E_9\rightarrow\int_0^t\rho_{s-}(\phi(\lambda-1))dZ_s.
	\end{aligned}
\end{equation*}

From the pointwise convergence of $\eta_t^{\varepsilon}$ to $\eta_t$ as $\varepsilon\to0$, it follows that {\color{blue}$\lim\limits_{\varepsilon\rightarrow 0}\eta_t^{\varepsilon}\phi=\eta_t\phi$}.  In addition,
\begin{equation*}
	\tilde{E}\|\eta_t^{\varepsilon}\phi\|\leq\|\phi\|_{\infty}\tilde{E}(\eta_t)=\|\phi\|_{\infty} E(\eta_t\eta_t^{-1})=\|\phi\|_{\infty}<\infty.
\end{equation*}
Due to the dominated convergence theorem \cite[page 152]{ReDY1999}, we obtain
\begin{equation*}
	\lim_{\varepsilon\rightarrow 0}\tilde{E}(\eta_t^{\varepsilon}\phi\vert  \mathcal{F}_T)=\tilde{E}(\eta_t\phi\vert \mathcal{F}_T)=\rho_t(\phi),\quad\tilde{P}-a.s.
\end{equation*}
Similarly, there holds
\begin{equation*}
	\lim_{\varepsilon\rightarrow 0}E_1=\tilde{E}(\eta_0\phi(X_0)\vert \mathcal{F}_T)=\rho_0(\phi),\quad\tilde{P}-a.s.
\end{equation*}

Now, we consider  item $E_2$. Notice that
{\color{blue} for any $\varepsilon>0$, there holds}
\begin{equation*}
	\tilde{E} \bigg(\|\int_0^t\eta_t^{\varepsilon}\mathcal{L}\phi ds\| \big\vert \mathcal{F}_T\bigg)=\tilde{E} \bigg(\|\int_0^t
	\frac{\varepsilon\eta_s}{1+\varepsilon\eta_s}\frac{1}{\varepsilon}\mathcal{L}\phi ds\|\big\vert \mathcal{F}_T\bigg)\leq \frac{1}{\varepsilon}\|\mathcal{L}\phi\|_{\infty}T<\infty.
\end{equation*}
By Fubini's theorem, we  exchange the integral order in $E_2$ to obtain
\begin{equation*}
	\begin{aligned}
		\|E_2\|&=\|\tilde{E} \bigg(\int_0^t\eta_s^{\varepsilon}\mathcal{L}\phi ds\rvert\mathcal{F}_T\bigg)\|=\|\int_0^t\tilde{E}(\eta_s^{\varepsilon}\mathcal{L}\phi\vert \mathcal{F}_T)ds\|\\
		& \leq
		\tilde{E} \bigg(\int_0^t\tilde{E}(\|\eta_t^{\varepsilon}\mathcal{L}\phi\|\rvert\mathcal{F}_T)ds\bigg)\leq{} \tilde{E}
		\bigg(\int_0^t\tilde{E}(\eta_s\|\mathcal{L}\phi\|_{\infty}\vert\mathcal{F}_T)ds\bigg)
		\\
		&=\|\mathcal{L}\phi\|_{\infty}\int_0^t\tilde{E}(\eta_s)ds
		\leq\|\mathcal{L}\phi\|_{\infty}{\color{blue}T}<\infty.
	\end{aligned}
\end{equation*}
By the dominated convergence theorem, we  get
\begin{equation*}
	\lim_{\varepsilon\rightarrow 0}E_2=\int_0^t\tilde{E}(\eta_s\mathcal{L}\phi\vert \mathcal{F}_T)ds=\int_0^t\rho_s(\mathcal{L}\phi)ds,\quad \tilde{P}-a.s.
\end{equation*}

Next, we study item $E_3$. Notice that 
\begin{equation*}
	\begin{aligned}
		\tilde{E} \bigg(\int_0^t\|\frac{\varepsilon\eta_s^2}{(1+\varepsilon\eta_s)^2}\nabla\phi^T  B_1h\|ds\bigg)\leq{}&\|\nabla\phi\|\tilde{E} \big(\int_0^t\eta_s \| B_1h\|ds \big)
		\\
		={}&\|\nabla\phi\|\int_0^tE(\| B_1h\|)ds<\infty.
	\end{aligned}
\end{equation*}
This estimate, together with 
 the  the pointwise convergence  $\lim\limits_{\varepsilon\rightarrow 0} \frac{\varepsilon\eta_t^2}{(1+\varepsilon\eta_t)^2}\nabla\phi^T  B_1h=0$ and the dominated convergence theorem implies
$$\lim\limits_{\varepsilon\rightarrow 0}E_3=0, \quad \tilde{P}-a.s.$$
In a similar way, we obtain $\lim\limits_{\varepsilon\rightarrow 0}E_i=0$, $\tilde{P}-a.s.$ for $i=6,7$.

By isometry formula, we have that  for any $\varepsilon>0$
\begin{equation*}
	\begin{aligned}
		&\tilde{E}\left((\int_0^t\eta_s^{\varepsilon}\nabla\phi^T
		B_1dY_s)^2\right)\\
		={}&\tilde{E} \bigg( \int_0^t\big(\frac{\eta_s}{1+\varepsilon\eta_s}\big)^2\nabla\phi^T  B_1D B_1^T\nabla\phi ds \bigg)\\
		={}&\frac{1}{\varepsilon}\tilde{E}\bigg( \int_0^t
		\frac{\varepsilon\eta_s}{(1+\varepsilon\eta_s)^2}\eta_s\nabla\phi^T B_1D B_1^T\nabla\phi ds \bigg)
		\\
		\leq{}&\frac{\|\nabla\phi\|^2}{\varepsilon}\tilde{E}\big( \int_0^t\eta_s\| B_1D B_1^T\|ds \big)
		={}\frac{\|\nabla\phi\|^2}{\varepsilon}\int_0^tE( \|B_1D B_1^T\|)ds<\infty.
	\end{aligned}
\end{equation*}
According to \cite[Lemma 3.21]{BaCr2009}, we can change the order between conditional expectation and stochastic integral to obtain
\begin{equation}
	\label{eq:2.1.12}
	E_4=\int_0^t\tilde{E}\big( \frac{\eta_s}{1+\varepsilon\eta_s}\nabla\phi^T  B_1\vert \mathcal{F}_T\big)dY_s.
\end{equation}
Using Jensen's inequality and Fubini's Theorem,  {\color{blue} for any $\varepsilon>0$}, we have
\begin{equation*}
	\begin{aligned}
		& \tilde{E}(E_4^2)
		\\
		=&
		\tilde{E}\left\{\int_{0}^{t}\big[\tilde{E}(\frac{\eta_s}{1+\varepsilon\eta_s}\nabla\phi^T  B_1\vert \mathcal{F}_T)\big]D
		\big[\tilde{E}(\frac{\eta_s}{1+\varepsilon\eta_s}\nabla\phi^T B_1\vert \mathcal{F}_T)\big]^Tds\right\}
		\\
		\leq
		&
		\tilde{E}\left\{\int_{0}^{t}\tilde{E}\big[(\frac{\eta_s}{1+\varepsilon\eta_s}\nabla\phi^T  B_1)D(\frac{\eta_s}{1+\varepsilon\eta_s}\nabla\phi^T  B_1)^T\vert \mathcal{F}_T\big]ds\right\}
		\\
		\leq&\tilde{E}\left\{\int_{0}^{t}\tilde{E}\big[\|\nabla\phi\|^2\|B_1D B_1^T\|\frac{1}{\varepsilon}\frac{\varepsilon\eta_s}{1+\varepsilon\eta_s}\eta_s
		\vert \mathcal{F}_T\big]ds\right\}
		\\
		\leq
		&
		\frac{\|\nabla\phi\|^{2}}{\varepsilon}\tilde{E}\bigg(\int_{0}^{t}\tilde{E}(\eta_s\|B_1DB_1^T\|)ds\bigg)
		=
		\frac{\|\nabla\phi\|^{2}}{\varepsilon}\int_{0}^{t}\tilde{E}(\eta_s\|B_1D B_1^T\|)ds
		\\
		=
		&
		\frac{\|\nabla\phi\|^{2}}{\varepsilon}\int_{0}^{t}E(\|B_1D B_1^T\|)ds<\infty.
	\end{aligned}
\end{equation*}
This implies that  $E_4$  is a martingale, c.f.\ \cite[Theorem 4.3.1]{ShSE2004} and then
the process $\int_0^t\rho_s(\nabla\phi^T  B_1)dY_s$
is a local martingale.  Thus, the following difference  is a local martingale
\begin{equation}
	\label{eq:2.1.14}
	\int_0^t\rho_s(\nabla\phi^T  B_1)dY_s-E_4=\int_0^t\tilde{E}\big( \frac{\varepsilon(\eta_s)^2}{1+\varepsilon\eta_s}\nabla\phi^T  B_1\vert \mathcal{F}_T\big)dY_s.
\end{equation}
Set  $\xi_s^{\varepsilon}=\frac{\varepsilon(\eta_s^2)}{1+\varepsilon\eta_s}\nabla\phi^T  B_1$, then $\lim\limits_{\varepsilon\rightarrow 0}\xi_s^{\varepsilon}=0$, $\tilde{P}-a.s.$ Obviously, for any $\phi\in V$  we have
\begin{equation*}
\|\xi_s^{\varepsilon}\|\leq \|\nabla\phi\| \| B_1\|\eta_s.
\end{equation*}
Due to the  dominated convergence theorem, we  obtain
\begin{equation*}
	\lim\limits_{\varepsilon\rightarrow 0}\tilde{E}(\xi_s^{\varepsilon}\vert \mathcal{F}_T)=\tilde{E}(	\lim\limits_{\varepsilon\rightarrow 0}\xi_s^{\varepsilon}\vert \mathcal{F}_T)=0,\tilde{P}-a.s.
\end{equation*}
Furthermore,
\begin{equation}
	\label{eq:2.1.15}
		\|\tilde{E}(\xi_s^{\varepsilon}\vert \mathcal{F}_T)\| \leq
		\tilde{E}(\|\xi_s^{\varepsilon}\| \vert \mathcal{F}_T) \leq  \|\nabla\phi\| \| B_1\| \tilde{E}{(\eta_s\vert \mathcal{F}_T)} \leq  \|\nabla\phi\| \| B_1\|.
\end{equation}
Applying the stochastic dominated convergence theorem \cite[Theorem 32]{Protter2005}, we have 
\begin{equation*}
	\lim\limits_{\varepsilon\rightarrow 0}\int_0^t\tilde{E}(\xi_s^{\varepsilon}\vert \mathcal{F}_T)dY_s=\int_0^t\lim\limits_{\varepsilon\rightarrow 0}\tilde{E}(\xi_s^{\varepsilon}\vert \mathcal{F}_T)dY_s=0,\quad\tilde{P}-a.s.
\end{equation*}
Hence, we obtain 
\begin{equation*}
	\lim\limits_{\varepsilon\rightarrow 0}E_4=\int_0^t\rho_s(\nabla\phi  B_1)dY_s,\quad \tilde{P}-a.s.
\end{equation*}

Similarly, we  can prove that 
\begin{equation*}
	\lim\limits_{\varepsilon\rightarrow 0}E_5=\int_0^t\rho_s(\phi h^TD^{-1})dY_s,\quad\lim\limits_{\varepsilon\rightarrow 0}E_8=\int_0^t\rho_s(\phi(\lambda-1))ds,\quad\tilde{P}-a.s.\\
\end{equation*}

Finally, we investigate  the term  $E_9$. Let
\begin{equation*}
	G_s^{\varepsilon}:=\frac{\eta_{s-}(\lambda(X_{s-})-1)\phi(X_{s-})}{(1+\varepsilon\eta_{s-}\lambda(X_{s-}))(1+\varepsilon\eta_{s-})}. 
\end{equation*}
Then $\lim\limits_{\varepsilon\rightarrow 0} G_s^{\varepsilon}=\eta_{s-}\big( \lambda(X_{s-})-1 \big)\phi(X_{s-})$, $\tilde{P}-a.s.$

  It is easy to see that for any $\varepsilon>0$
\begin{align}
  \label{eq:2.2.18b}
  \vert G_s^{\varepsilon}\vert \leq \frac{\|\phi\|_{\infty}\|\lambda-1\|_{\infty}}{\varepsilon}<\infty.
\end{align}
This  estimate, together with  the stochastic Fubini's theorem  \cite[Theorem 64]{Protter2005} implies that we can change the order between the stochastic integral and
  conditional expectation in $E_9$ to obtain $E_9=\int_0^t\tilde{E}(G_s^{\varepsilon}\vert \mathcal{F}_T)dZ_s$.

Using the same argument as above, we obtain
\begin{equation*}
\lim \limits_{\varepsilon\rightarrow 0} E_9 =\int_0^t\tilde{E}\big( \phi(X_{s-})\eta_{s-}(\lambda(X_{s-})-1)\vert \mathcal{F}_T \big)dZ_s
\end{equation*}

According to  \cite[Theorem 1.6]{Liptser2001}, we have
\begin{equation}
	\label{eq:2.1.18}
	\begin{aligned}
		&\tilde{E}\big( \phi(X_{s-})\eta_{s-}(\lambda(X_{s-})-1)\vert \mathcal{F}_{T}\big)
		=\lim\limits_{r\uparrow s}\tilde{E}\big( \phi(X_r)\eta_r(\lambda(X_r)-1)\vert \mathcal{F}_T\big)\\
		&=\lim\limits_{r\uparrow s}\rho_r \big( \phi(\lambda-1) \big)=
		\rho_{s-}\big( \phi(\lambda-1)\big)
	\end{aligned}
\end{equation}
Hence, $\displaystyle\lim\limits_{\varepsilon\rightarrow 0}E_9=\int_0^t\rho_{s-}\big( \phi(\lambda-1) \big)dZ_s,~\tilde{P}-a.s.$. 
\end{proof}


The next theorem  follows from Theorem \ref{thm:2.1.1}.

\begin{theorem}
	\label{thm:2.1.2}
	Assume {\bf H1}-{\bf H3}. Then $p(t)$ satisfies Zakai equation:
	\begin{equation}
		\label{eq:2.1.20}
		dp(t)=\mathcal{L}^{\star}p(t)dt+\mathcal{B}^{\star}p(t)dY_t+\mathcal{C}p(t-)(dZ_t-dt),\quad p(0)=p_0\in H,
	\end{equation}
	where for any $\phi\in V$
	\begin{eqnarray*}
		\mathcal{L}^{\star}\phi
		&=&
		\sum\limits_{i,j=1}^{d}\frac{\partial^2}{\partial x_{i}\partial x_{j}}(a_{i,j}\phi)-\sum\limits_{i=1}^{d}\frac{\partial}{\partial x_{i}}(g_{i}\phi),
		\\
		\mathcal{B}^{\star}\phi
		&=&
		\phi h^TD^{-1}-\sum\limits_{i=1}^{d}\frac{\partial}{\partial x_{i}}(b_{1,i}\phi),\\
		\mathcal{C}\phi&=&(\lambda-1)\phi.
	\end{eqnarray*}
\end{theorem}

The differential operator $\mathcal{B}^{\star}$ is not bounded in the usual sense. We now study its boundedness in $L(V,V^\prime)$.   Due to {\bf H2}, for any $\phi,\psi\in V$,  $\mathcal{B}^{\star}\phi\in H\subset V^\prime$
and satisfies
\begin{equation}
	\label{eq:2.1.22}
	\|\mathcal{B}^{\star}\phi\|_{V^{\prime}}=
	\sup_{\|\psi\|=1}(\phi,\mathcal{B}\psi)\leq \sup_{\|\psi\|=1} \|\phi\| \|\mathcal{B} \psi\|
	\leq C\|\phi\|,
\end{equation}
where $C>0$ is a constant.

The following lemma is concerned with  the regularity of  the second-order differential operator $-\mathcal{L}^{\star}+\mathcal{C}$.

\begin{lemma}
	\label{lem:2.2.5}
	Assume  {\bf H1}-{\bf H3}. Then there  exist constants $\beta_1,\beta_2>0$ and $\alpha> 0$
	such that $\forall\phi,\psi\in V$
	\begin{equation*}
		\begin{aligned}
		&
		\vert\langle(\mathcal{L}^{\star}-\mathcal{C})\phi,\psi\rangle\vert\leq\beta_1\|\phi\|_1\|\psi\|_1,
		\\
		&
		\beta_2\|\phi\|_1^2  \leq -\langle(\mathcal{L}^{\star}-\mathcal{C})\phi,\phi\rangle+\alpha\|\phi\|^2.
		\end{aligned}
	\end{equation*}
\end{lemma}

\begin{proof}
The first inequality directly follows from assumption {\bf H2}.  Thus we only need to prove the second one.
	
	Direct computation gives, for any $\phi\in V$,
	\begin{equation*}
		\mathcal{L}^{\star}\phi
		=\sum\limits_{i,j=1}^d \frac{\partial}{\partial x_i}(a_{i,j}\frac{\partial \phi}{\partial x_j})-\sum\limits_{i=1}^d \frac{\partial}{\partial x_i}(\delta_i\phi)
	\end{equation*}
	where $\delta$ is a $\mathbb{R}^d$-valued function with $\delta_i=g_i-\sum\limits_{j=1}^d \frac{\partial a_{i,j}}{\partial x_j}$. Therefore,
	$$
	-\langle\mathcal{L}^{\star}\phi,\phi\rangle = (A\nabla\phi,\nabla\phi)-(\delta\phi,\nabla\phi).
	$$
	
	From assumption {\bf H3}, it follows that ${\|\delta\|}<\infty$. Take a positive number $\alpha>-(\varpi_2-1-\frac{1}{2\alpha_1}\|\delta\|^2)$ and let   $\beta_2 = \min(\frac12\alpha_1,\alpha+\varpi_2-1-\frac{1}{2\alpha_1}\|\delta\|^2)$, then we obtain 
	\begin{align*}
		-\langle (\mathcal{L}^{\star}-\mathcal{C})\phi,\phi\rangle +\alpha\|\phi\|^2
		&\geq \alpha_1\|\nabla\phi\|^2-\langle \delta\phi, \nabla\phi\rangle +(\alpha+\varpi_2-1)\|\phi\|^2
		\\
		&\geq \frac12\alpha_1\|\nabla\phi\|^2+(\alpha+\varpi_2-1-\frac{1}{2\alpha_1}\|\delta\|^2)\|\phi\|^2
		\\
		& \geq \beta_2\|\phi\|_1^2.
	\end{align*}
\end{proof}


The existence, uniqueness,  and regularity of the solution to the Zakai equation (\ref{eq:2.1.20}), which we summarize in the next lemma.,  can be obtained following the approaches in \cite{BenA1987,Pardoux1979}, 

\begin{lemma}\label{t1.5.1}
	Assume  {\bf H1}-{\bf H3}. For each $p_0\in H$, there exists a unique solution $p$ of
	\eqref{eq:2.1.20} satisfying
	\begin{equation*}
		p\in \mathbb{L}^2(0,T;V)\bigcap L^2(\Omega;C(0,T;H)).
	\end{equation*}
	Furthermore, if $p_0\in  L^4(\Omega;H)$, there holds $p\in L^{\infty}(0,T;L^{4}(\Omega;H))$.
\end{lemma}

In order to construct a stable and  efficient numerical method, we
take a  transformation $p(t)\to p(t)e^{-\mu t}$ in  \eqref{eq:2.1.20}
and obtain  a well-posed Zakai equation
\begin{equation}
	\label{eq:2.1.23}
	dp(t)+(-\mathcal{L}^{\star}p(t)+\mu p(t))dt=\mathcal{B}^{\star}p(t)dY_t+\mathcal{C}p(t-)(dZ_t-dt),
	\quad p(0)=p_0\in H.
\end{equation}

\section{A splitting-up  scheme and error estimation}\label{sec3}
In this section, we apply a splitting-up method to construct a temporal semi-discretized    approximation  of  equation \eqref{eq:2.1.23} and derive corresponding error estimations.

\subsection{A splitting-up approximate solution} 

Consider   three equations
\begin{align}
	dp_1(t)+\frac{\mu}{3}p_1(t)dt={}&\mathcal{B}^{\star}p_1(t)dY_t,\label{eq:2.2.1}
	\\
	dp_2(t)+\frac{\mu}{3}p_2(t)dt={}&(\mathcal{L}^{\star}-\mathcal{C})p_2(t)dt, \label{eq:2.2.1-1}
	\\
	dp_3(t)+\frac{\mu}{3}p_3(t)dt={}&\mathcal{C}p_3(t-)dZ_t. \label{eq:2.2.1-2}
\end{align}
Equation \eqref{eq:2.2.1} is a first-order SPDE.  We denote by   $\{Q_t^s, 0\leq s\leq t\}$ with  $Q_s^s=I$  its solution operator which is a Markov
semigroup, cf \cite{DaKL2004}.
\eqref{eq:2.2.1-1} is a  determined second-order PDE satisfying uniform elliptic condition if $\mu$ is large.  Denote by  $\{R_t^s, t\geq 0\}$ with $R_s^s=I$  its  strongly continuous semigroup.  \eqref{eq:2.2.1-2} is a stochastic differential equation driven by a point process. The existence and uniqueness of the solution to  \eqref{eq:2.2.1-2}  are obtained in   \cite{Protter2005}. We denote its solution operator by  $\{\Gamma_t^s, t\geq 0\}$ with  $\Gamma_s^s=I$.   According to \cite{GLMV2011,Protter2005}, there exists a
constant $C=C(T)$ such that for any $\phi\in L^2(\Omega;H)$ and  $0\leq s<t\leq T$, there holds
\begin{equation}\label{eq:2.2.19-2}
	\tilde{E}\|Q_t^s  \phi\|\leq C \tilde{E}\|\phi\|,\quad  \tilde{E}\|R_t^s  \phi\|\leq C \tilde{E}\|\phi\|,\quad \tilde{E}\|\Gamma_t^s \phi\|\leq C \tilde{E}\|\phi\|.
\end{equation}

For any given integer
$N>0$, let $t_r=r\kappa$ 
$(r=0,1,\cdots,N)$  be the uniformly partition of  interval $[0,T]$ with stepsize $\kappa=\frac{T}{N}$.  By virtue of the solutions of \eqref{eq:2.2.1}-\eqref{eq:2.2.1-2} in each interval $[t_r,t_{r+1}]$ we define a splitting-up  solution $p_{\kappa}^{r+1}$ to  \eqref{eq:2.1.23}   at each node point $t_{r+1}$   by
\begin{equation}\label{eq:2.4.6-1}
	p_{\kappa}^{r+1}=\Gamma_{t_{r+1}}^{t_r}R_{t_{r+1}}^{t_r} Q_{t_{r+1}}^{t_r}p_{\kappa}^r,\quad r=0,1,\cdots, N-1,\quad p_\kappa^0=p_0.
\end{equation}
Meanwhile, we also define three solution process to equations \eqref{eq:2.2.1}-\eqref{eq:2.2.1-2}  in each interval $[t_r,t_{r+1}]$,
\begin{equation*}
	p_{1\kappa}(t)=Q^{t_r}_{t}p_\kappa^r,\quad
	p_{2\kappa}(t)=R^{t_r}_{t}Q^{t_r}_{t_{r+1}}p_\kappa^r,\quad
	p_{3\kappa}(t)=\Gamma^{t_r}_{t}R^{t_r}_{t_{r+1}}Q^{t_{r}}_{t_{r+1}}p_\kappa^r.
\end{equation*}
Obviously, $p_{\kappa}^{r+1}=p_{3\kappa}(t_{r+1})$.  Then we have that
$p_{1\kappa},\, p_{3\kappa}\in\mathbb{L}^2(0,T;H)$ and $p_{2\kappa}\in\mathbb{L}^2(0,T;V)$ are right continuous
and their discontinuity only occurs at  node points.
Furthermore, there hold
\begin{equation}
	\label{eq:2.2.4}
	\begin{aligned}
		&p_{1k}(t)~\text{is}~\mathcal{F}_t~\text{measurable for }~ t\in[t_r,t_{r+1}],
		\\
		&p_{2k}(t)~\text{and}~p_{3k}(t)~\text{are}~\mathcal{F}_{t_{r+1}}~\text{measurable for}~ t\in[t_r,t_{r+1}].
	\end{aligned}
\end{equation}

\subsection{\textit{A priori} estimate}

\begin{theorem}
	\label{thm:2.2.1}
	Assume {\bf H1}-{\bf H3}. Then there exist constants  $\mu>0$ and $C=C(T)>0$ such that the  three processes $p_{i\kappa}(t)$ $(i=1,2,3)$ satisfy
	\begin{equation}
		\label{eq:2.2.5}
		\tilde{E}\int_0^T\|p_{1\kappa}\|^2dt\leq C, \quad \tilde{E}\int_0^T\|p_{2\kappa}\|_1^2dt\leq C, \quad \tilde{E}\int_0^T\|p_{3\kappa}\|^2dt\leq C.
	\end{equation}
	Furthermore, if $p_0\in L^4(\Omega;H)$, then there hold
	\begin{equation}\label{eq:2.2.13}
		\tilde{E}\|p_{1\kappa}(t)\|^4\leq C, \quad \tilde{E}\|p_{2\kappa}(t)\|^4\leq C, \quad \tilde{E}\|p_{3\kappa}(t)\|^4\leq C, \quad \forall t\in[0,T].
	\end{equation}
\end{theorem}

\begin{proof}
	On each interval $(t_r,t_{r+1})$, by  It\^o formula  there holds
	\begin{align}
		&d\|p_{1\kappa}(t)\|^2+(\frac{2}{3}\mu\|p_{1\kappa}(t)\|^2-\vert\mathcal{B}^{\star}p_{1\kappa}(t)\vert_{D}^2)dt=2(p_{1\kappa}(t),\mathcal{B}p_{1\kappa}(t))dY_t,\label{eq:2.2.8}
		\\
		& d\|p_{2\kappa}(t)\|^2+(\frac{2}{3}\mu\|p_{2\kappa}(t)\|^2-2\langle (\mathcal{L}^{\star}-\mathcal{C})p_{2\kappa}(t),p_{2\kappa}(t)\rangle)dt = 0,\label{eq:2.2.7}
		\\
		&d\|p_{3\kappa}(t)\|^2+\frac{2}{3}\mu\|p_{3\kappa}(t)\|^2dt= (\lambda^2-1)\|p_{3\kappa}(t-)\|^2dZ_t,\label{eq:2.2.9}
	\end{align}
where $\vert\mathcal{B}^{\star}p(t)\vert_D^2=\int_{\mathbb{R}^d}[\mathcal{B}^{\star}p(t)]D[\mathcal{B}^{\star}p(t)]^Tdx\leq M\|p\|^2$ due to  (\ref{eq:2.1.22}) and assumption {\bf
          H3}. 
	
	Choose $\mu = \max\{3\alpha, \frac{3}{2} M\}+1$ and $\gamma = \min\{\frac{2}{3}\mu-M,2\beta_2,\frac{2}{3}\mu\}$.
	By Lemma \ref{lem:2.2.5}, we obtain
	\begin{equation*}
		\begin{aligned}
			d(\|p_{1\kappa}\|^2+\|p_{2\kappa}\|^2+\|p_{3\kappa}\|^2)+\gamma(\|p_{1\kappa}\|^2+\|p_{2\kappa}\|_1^2+\|p_{3\kappa}\|^2)dt& \\
			\leq2(p_{2\kappa},\mathcal{B}p_{2\kappa})dY_t+(\lambda^2-1)\|p_{3\kappa}(t-)\|^2dZ_t.&
		\end{aligned}
	\end{equation*}
	Integrating this equation over  $(t_i,t_{i+1})$ and taking expectation yields
	\begin{equation*}
		\begin{aligned}
			\tilde{E}\big(\|p_{1\kappa}(t_{i+1}-)\|^2+\|p_{2\kappa}(t_{i+1}-)\|^2+\|p_{3\kappa}(t_{i+1}-)\|^2\big)-
			\tilde{E}\big(\|p_{1\kappa}(t_i)\|^2+\|p_{2\kappa}(t_i)\|^2+&
			\\
			\|p_{3\kappa}(t_i)\|^2\big)+\gamma \tilde{E}\int_{t_i}^{t_{i+1}}\big(\|p_{1\kappa}\|^2+\|p_{2\kappa}\|_1^2+\|p_{3\kappa}\|^2\big)dt\leq C\tilde{E}\int_{t_i}^{t_{i+1}}\|p_{3k}(t-)\|^2dZ_t.&
		\end{aligned}
	\end{equation*}
	Together with (\ref{eq:2.2.1})-(\ref{eq:2.2.1-2}), we get
	\begin{equation*}
		\tilde{E}\|p_\kappa^{i+1}\|^2\!-\!\tilde E\|p_\kappa^i\|^2\!+\!\gamma \tilde{E}\int_{t_i}^{t_{i+1}}\!\!(\|p_{1\kappa}\|^2\!+\!\|p_{2\kappa}\|_1^2\!+\!\|p_{3\kappa}\|^2)dt
		\!\leq\! C\tilde{E}\int_{t_i}^{t_{i+1}}\!\!\|p_{3\kappa}(t-)\|^2dZ_t.
	\end{equation*}
	Summing  up  this equation from $i=0$ to $r-1$ gives
	\begin{equation} \label{eq:2.2.12-1}
		\tilde{E}\|p_\kappa^{r}\|^2\!+\!\gamma \tilde{E}\int_{0}^{t_r}\!\!(\|p_{1\kappa}\|^2\!+\!\|p_{2\kappa}\|_1^2\!+\!\|p_{3\kappa}\|^2)dt
		\!\leq\! \!\tilde E\|p_0\|^2+C\tilde{E}\int_{0}^{t_r}\!\!\|p_{3\kappa}(t-)\|^2dZ_t.	
	\end{equation}
	Since the Poisson process $Z_t$ has only finite jump times, we have
	$$
	\tilde E\int_0^{T}\|p_{3k}(t-)\|^2dZ_t\leq C.
	$$
	Thus, we obtain $\tilde{E}\|p_\kappa^r\|^2\leq C$,  $r=1,\cdots,N$, and estimates  \eqref{eq:2.2.5}
         follow from \eqref{eq:2.2.12-1}.
        
	Next, we integrate (\ref{eq:2.2.8}) and (\ref{eq:2.2.7}) over $(t_r,t_{r+1})$ and then take expectation to obtain
	\begin{equation*}
		\tilde{E}\|p_{1\kappa}(t_{r+1}-0)\|^2\leq \tilde{E}\|p_{\kappa}^r\|^2\leq C, \quad\tilde{E}\|p_{2\kappa}(t_{r+1}-0)\|^2\leq \tilde{E}\|p_{1\kappa}(t_{r+1}-0)\|^2\leq C.
	\end{equation*}
	
	For any $t\in (t_r,t_{r+1})$, integrating (\ref{eq:2.2.8}), (\ref{eq:2.2.7}) and (\ref{eq:2.2.9}) over $[t_r, t]$,  we obtain
	\begin{equation*}
		\begin{aligned}
			&\tilde{E}\|p_{1\kappa}(t)\|^2\leq \tilde{E}\|p_\kappa^r\|^2\leq C,
			\quad \tilde{E}\|p_{2\kappa}(t)\|^2\leq \tilde{E}\|p_{1\kappa}(t_{r+1}-0)\|^2\leq C,
			\\
			&\tilde{E}\|p_{3\kappa}(t)\|^2\leq \tilde{E}\|p_{2\kappa}(t_{r+1}-0)\|^2+\tilde{E}\int_{t_r}^t(\lambda^2-1)\|p_{3\kappa}(t-)\|^2dZ_t\leq C,
		\end{aligned}
	\end{equation*}
	where the constant $C$ is independent of $r$.
	Therefore,
	\begin{equation}
		\label{eq:2.2.14}
		\tilde{E}\|p_{i\kappa}(t)\|^2\leq C,\quad\forall t\in [0,T],\quad i=1,2,3.
	\end{equation}
	
	By Ito's formula,  (\ref{eq:2.2.8}), (\ref{eq:2.2.7}) and (\ref{eq:2.2.9}),  we have for any $t\in [t_r,t_{r+1}]$, $r=0, 1, 2,\cdots, N-1$
	\begin{align*}
		&d\|p_{1\kappa}\|^4+[2\|p_{1\kappa}\|^2(\frac{2}{3}\mu\|p_{1\kappa}\|^2-\vert\mathcal{B}^{\star}p_{1\kappa}\vert_{D}^2)-4\vert\mathcal{B}^{\star}p_{1\kappa}\vert_D^4]dt = 4\|p_{1\kappa}\|^2(p_{1\kappa}, \mathcal{B}p_{1\kappa})dY_t,
		\\
		&d\|p_{2\kappa}\|^4+2\|p_{2\kappa}\|^2(\frac{2}{3}\mu\|p_{2\kappa}\|^2-2\langle (\mathcal{L}^{\star}-\mathcal{C})p_{2\kappa},p_{2\kappa}\rangle)dt=0,
		\\
		&d\|p_{3\kappa}\|^4+\frac{4}{3}\mu\|p_{3\kappa}\|^4dt=(\lambda^4-1)\|p_{3\kappa}(t-)\|^4dZ_t.
	\end{align*}
	Following the same argument above, we can also derive the estimates \eqref{eq:2.2.13}.
\end{proof}

\subsection{Convergence of splitting-up  solution}
In this section, we shall investigate the convergence and convergence order of the splitting-up  solution.

\begin{theorem}
	\label{thm:2.3.1}
	Assume that {\bf H1}-{\bf H3} hold and $p_0\in L^4(\Omega;H)$. Then as $\kappa\to 0$, there holds
	\begin{equation*}
		\begin{aligned}
			&p_{1\kappa}(t),p_{2\kappa}(t),p_{3\kappa}(t)\rightarrow p(t)~\text{in}~L^2(\Omega;H), ~\text{uniformly for }~ t\in [0,T],
			\\
			&p_{1\kappa},\,p_{3\kappa}\rightarrow p~\text{in}~\mathbb{L}^2(0,T;H),\quad p_{2\kappa}\rightarrow p~\text{in}~\mathbb{L}^2(0,T;V).
		\end{aligned}
	\end{equation*}
\end{theorem}

Before proving the theorem, we notice that, according to  Theorem \ref{thm:2.2.1}, the three sequences $p_{1\kappa}(t)$, $p_{2\kappa}(t)$ and $p_{3\kappa}(t)$ are bounded in  spaces
$\mathbb{L}^2(0,T;H)$, $\mathbb{L}^2(0,T;V)$ and $\mathbb{L}^2(0,T;H)$, respectively. By the weakly compactness of these spaces, we can extract three subsequences, still denoted by $p_{1\kappa}(t)$, $p_{2\kappa}(t)$ and $p_{3\kappa}(t)$, such that as $\kappa\to 0$
\begin{equation}
	\label{eq:2.3.1}
	\begin{aligned}
		&(p_{1\kappa},p_{3\kappa})\rightarrow (p_1,p_3)~\text{in}~ \mathbb{L}^2(0,T;H)~\text{weakly},
		\\
		& p_{2\kappa}\rightarrow p_2~ \text{in}~ \mathbb{L}^2(0,T;V)~\text{weakly},
		\\
		&(p_{i\kappa}-p_{1\kappa})\to (p_i-p_1) ~\text{in}~\mathbb{L}^2(0,T;V^\prime)~\text{weak star for }~ i=2,3.
	\end{aligned}
\end{equation}
Furthermore, if $p_0\in L^4(\Omega;H)$
\begin{equation}\label{eq:2.2.19-1}
	p_{i\kappa}\rightarrow p_i~ \text{in}~ L^{\infty}(0,T;L^4(\Omega;H))~\text{for}~ i=1,2,3 \text{~weak star}~\text{as}~\kappa\to0.
\end{equation}

To prove the theorem, we need a series of lemmas. 
\begin{lemma}
	\label{lem:2.3.1}
	Assume that {\bf H1}-{\bf H3} hold and $p_0\in L^4(\Omega;H)$. Then
	the functions $p_1$, $p_2
	$ and $p_3$ are equal to a common function  $\xi \in \mathbb{L}^2(0,T;V)\cap L^{\infty}(0,T;L^4(\Omega;H))$.
\end{lemma}
\begin{proof}
	Integrating (\ref{eq:2.2.1}) over $(t,t_{r+1})$ , \eqref{eq:2.2.1-1} over $(t_r,t_{r+1})$ and   (\ref{eq:2.2.1-2}) over   $(t_r,t)$, yields
	\begin{equation}\label{eq:2.3.1-1}
		\begin{aligned}
			&p_{1\kappa}(t_{r+1}-0)-p_{1\kappa}(t)+\int_{t}^{t_{r+1}} \frac{\mu}{3}p_{1\kappa}(s)ds=\int_t^{t_{r+1}}\mathcal{B}^{\star}p_{1\kappa}dY_s,
			\\
			&p_{2\kappa}(t_{r+1}-0)-p_{1\kappa}(t_{r+1}-0)+\int_{t_r}^{t_{r+1}}\frac{\mu}{3}p_{2\kappa}(s)ds=\int_{t_r}^{t_{r+1}}(\mathcal{L}^{\star}-\mathcal{C})p_{2\kappa}(s)ds,
			\\
			&p_{3\kappa}(t)-p_{2\kappa}(t_{r+1}-0)+\int_{t_r}^t \frac{\mu}{3}p_{3\kappa}(s)ds=\int_{t_r}^t \mathcal{C}p_{3\kappa}(s-)dZ_s.
		\end{aligned}
	\end{equation}
	Adding them up, we get
	\begin{equation*}
		\begin{aligned}
			&p_{3\kappa}(t)-p_{1\kappa}(t)+\int_t^{t_{r+1}}\frac{\mu}{3}p_{1\kappa}(s)ds+\int_{t_r}^{t_{r+1}}(\frac{\mu}{3}p_{2\kappa}(s)-(\mathcal{L}^{\star}-\mathcal{C})p_{2\kappa}(s))ds
			\\
			&+\int_{t_r}^t \frac{\mu}{3}p_{3\kappa}(s)ds=\int_t^{t_{r+1}}\mathcal{B}^{\star}p_{1\kappa}(s)dY_s+\int_{t_r}^t\mathcal{C}p_{3\kappa}(s-)dZ_s.
		\end{aligned}
	\end{equation*}
	Then we have,  for $t\in[t_r,t_{r+1}]$, 
	\begin{equation*} 
		\begin{aligned}
			&\tilde{E}\|p_{3\kappa}(t)-p_{1\kappa}(t)\|_{V^{\prime}}^2
			\\
			&\leq
			5\tilde{E}\big(\int_{t}^{t_{r+1}}\!\!\frac{\mu}{3}\|p_{1\kappa}(s)\|_{V^{\prime}}ds\big)^2\!+\!
			5\tilde{E}\big(\int_{t_r}^{t_{r+1}}\!\!\big\|\frac{\mu}{3}p_{2\kappa}(s)\!-\!(\mathcal{L}^{\star}\!-\!\mathcal{C})p_{2\kappa}(s)\big\|_{V^\prime}ds\big)^2\\
			& +5\tilde{E}\big(\!\int_{t_r}^t
			\!\!\frac{\mu}{3}\|p_{3\kappa}(s)\|_{V^{\prime}}ds\big)^2\!+\!
			5\tilde{E}\big\|\!\int_{t_r}^{t_{r+1}}\!\!\mathcal{B}^{\star}p_{1\kappa}(s)dY_s\big\|_{V^{\prime}}^2
			\!+\!5\tilde{E}\big\|\int_{t_r}^t\!\!\mathcal{C}p_{3\kappa}(s-)dZ_s\big\|^2_{V^\prime}
			\\
			&
			:= 5 I_1+5I_2+5I_3+5I_4+5I_5.
		\end{aligned}
	\end{equation*}
	Due to Theorem \ref{thm:2.2.1} and \eqref{eq:2.1.2-1}, we have 
	\begin{equation*}
		I_1\leq \tilde{E}(\int_t^{t_{r+1}}\frac{\mu}{3}\|p_{1\kappa}(s)\|ds)^2\leq C\kappa\tilde{E}\int_t^{t_{r+1}}\|p_{1\kappa}(s)\|^2ds\leq C\kappa.
	\end{equation*}
	Similarly, we get
	\begin{align*}
		&I_2\leq C\kappa \tilde{E}\int_{t_r}^{t_{r+1}}\|p_{2\kappa}(s)\|_1^2ds\leq C\kappa,\quad I_3\leq C\kappa \tilde{E}\int_{t_r}^t\|p_{3\kappa}(s)\|^2ds\leq C\kappa.
	\end{align*}
	Applying  It\^o isometry formula to $I_4$, we have 
	\begin{equation*}
		\begin{aligned}
			I_4&
			= \int_{t_r}^{t_{r+1}}\tilde{E}\|\mathcal{B}^{\star}p_{1\kappa}(s)\|_{V^{\prime}}^2ds
			\leq C\int_{t_r}^{t_{r+1}}\tilde{E}\|p_{1\kappa}(s)\|^2ds\\
			&\leq C\kappa^{1/2}(\int_{t_r}^{t_{r+1}}\tilde{E}\|p_{1\kappa}(s)\|^4ds)^{1/2}
			\leq C\kappa^{1/2}.
		\end{aligned}
	\end{equation*}
	Since $Z_t-t$ is a martingale under  measure $\tilde{P}$,  we have
	\begin{align*}
		I_5&\leq \tilde{E}\big\|\int_{t_r}^t\!\!\mathcal{C}p_{3\kappa}(s-)d(Z_s-s)\big\|^2_{V^\prime}+\tilde{E}\big\|\int_{t_r}^t\!\!\mathcal{C}p_{3\kappa}(s-)ds\big\|^2_{V^\prime}
		\\
		&= \tilde{E}\int_{t_r}^t\|\mathcal{C}p_{3\kappa}(s-)\|^2ds+\tilde{E}\big\|\int_{t_r}^{t}\mathcal{C}p_{3\kappa}(s-)ds\big\|^2
		\\
		&\leq C\kappa^{1/2}(\tilde{E}\int_{t_r}^t\|p_{3\kappa}(s-)\|^4ds)^{1/2}+C\kappa\tilde{E}\int_{t_r}^t\|p_{3\kappa}(s-)\|^2ds
		\\
		&\leq C\kappa^{1/2}.
	\end{align*}
Therefore,  we have proved
	\begin{equation}\label{eq:2.3.2-0}
		\tilde{E}\|p_{3\kappa}(t)-p_{1\kappa}(t)\|_{V^{\prime}}^2\leq C\kappa^{1/2},\quad t\in[t_r,t_{r+1}].
	\end{equation}
This estimate leads to 
	\[
	\lim_{\kappa\to0} \tilde E\|p_{3\kappa}-p_{1\kappa}\|_{V^\prime}=0,~\text{uniformly~for}~t\in[0,T].
	\]
Thus we have proved  $p_3=p_1$.
Similarly we prove  $p_2=p_1$. Thus $\xi=p_1=p_2=p_3$, which completes the proof.
\end{proof}

\begin{lemma}
	\label{lem:2.3.2}
	Assume  that {\bf H1}-{\bf H3} hold  and $p_0\in L^4(\Omega;H)$. Then
	$p=\xi$ is the unique solution of \eqref{eq:2.1.23}.
\end{lemma}

\begin{proof}
	Integrating  equations (\ref{eq:2.2.1})-(\ref{eq:2.2.1-2}) over $(t_{i-1},t_i)$ and adding up, we get
	\begin{equation}
		\label{eq:2.3.4}
		\begin{aligned}
			&p_\kappa^{i}-p_\kappa^{i-1}+\int_{t_{i-1}}^{t_i}\big( \frac{\mu}{3}(p_{1\kappa}(s)+p_{2\kappa}(s)+p_{3\kappa}(s))-(\mathcal{L}^{\star}-\mathcal{C})p_{1k}(s)\big)ds
			\\
			&=\int^{t_i}_{t_{i-1}}\mathcal{B}^{\star}p_{1\kappa}(s)dY_s+\int^{t_i}_{t_{i-1}}\mathcal{C}p_{3\kappa}(s-)dZ_s.
		\end{aligned}
	\end{equation}
	Sum up this equation  from $i=0$ to $r$, we get
	\begin{equation}
		\label{eq:2.3.5}
		\begin{aligned}
			&p_\kappa^r-p_{\kappa}^0+\int_0^{t_r}\big( \frac{\mu}{3}(p_{1\kappa}(s)+p_{2\kappa}(s)+p_{3\kappa}(s))-(\mathcal{L}^{\star}-\mathcal{C})p_{1\kappa}(s)
			\big)ds
			\\
			&=\int_0^{t_r}\mathcal{B}^{\star}p_{1\kappa}dY_s+\int_0^{t_r}\mathcal{C}p_{3\kappa}(s-)dZ_s.
		\end{aligned}
	\end{equation}
	
	For any $t\in(t_r,t_{r+1})$,
	integrating (\ref{eq:2.2.1}) on $[t_r,t]$ leads to
	\begin{equation}
		\label{eq:2.3.6}
		p_{1\kappa}(t)-p_\kappa^r+\int_{t_r}^t \frac{\mu}{3}p_{1\kappa}(s)ds=\int_{t_r}^t\mathcal{B}^{\star}p_{1\kappa}(s)dY_s.
	\end{equation}
	
	Adding up  (\ref{eq:2.3.5}) and (\ref{eq:2.3.6}), we have
	\begin{equation}
		\label{eq:2.3.7}
		\begin{aligned}
			&p_{1\kappa}(t)-p_{\kappa}^0+\int_0^t \frac{\mu}{3}p_{1\kappa}(s)ds-\int_0^{[t/\kappa]\kappa}(\mathcal{L}^{\star}-\mathcal{C})p_{2\kappa}(s) ds \\
			&+\int_0^{[t/\kappa]\kappa} \frac{\mu}{3}(p_{2\kappa}(s)+p_{3\kappa}(s))ds=\int_0^t\mathcal{B}^{\star}p_{1\kappa}(s)dY_s+\int_0^{[t/\kappa]\kappa}\mathcal{C}p_{3\kappa}(s-)dZ_s.
		\end{aligned}
	\end{equation}
	Noticing that as $\kappa\to0$, for $i=2,3$
	\begin{equation*}
		\tilde{E}\bigg\|\int_{[t/\kappa]\kappa}^tp_{i\kappa}(s)ds\bigg\|^2\leq{}\big( t-[t/\kappa]\kappa \big)\big( \int_{[t/\kappa]\kappa}^t \tilde{E}\|p_{i\kappa}(s)\|^2ds \big)\leq C\big( t-[t/\kappa]\kappa \big)\rightarrow 0.
	\end{equation*}
	According to It\^o isometry formula, we have, {\color{blue} as $\kappa\to0$, } 
	\begin{equation*}
		\begin{aligned}
			&\tilde E\|\int_{[t/\kappa]\kappa}^t\mathcal{B}^{\star}p_{1\kappa}(s)dY_s\|_{V^\prime}^2
			=\tilde E \int_{[t/\kappa]\kappa}^t \|\mathcal{B}^{\star}p_{1\kappa}(s)\|_{V^\prime}^2ds
			\\
			\leq &(t-[t/\kappa]\kappa)^{1/2}(\tilde{E}\int_{[t/\kappa]\kappa}^t\|p_{1\kappa}(s)\|^4ds)^{1/2}\to 0,
		\end{aligned}
	\end{equation*}
	and 
	\begin{equation*}
		\begin{aligned}
			&\tilde{E}\left\|\int_{[t/\kappa]\kappa}^t\!\!\!\mathcal{C}p_{3\kappa}(s-)dZ_s \right\|^2
			\!\leq\! 2\tilde{E}\left\|\int_{[t/\kappa]\kappa}^t\!\!\!\mathcal{C}p_{3\kappa}(s-)(dZ_s\!-\!ds)\right\|^2
			\!\!+\!2\tilde{E}\left\|\int_{[t/\kappa]\kappa}^t\!\!\!\mathcal{C}p_{3\kappa}(s-)ds\right\|^2
			\\
			&= 2\tilde{E}\int_{[t/\kappa]\kappa}^t\|\mathcal{C}p_{3\kappa}(s-)\|^2ds+2\tilde{E}\left\|\int_{[t/\kappa]\kappa}^t\mathcal{C}p_{3\kappa}(s-)ds\right\|^2
			\\
			&\leq  C(t\!-\![t/\kappa]\kappa)^{1/2}\tilde{E}(\int_{[t/\kappa]\kappa}^t\!\!\|p_{3\kappa}(s-)\|^4ds)^{1/2}
			\!+\!C(t\!-\![t/\kappa]\kappa)\tilde{E}\int_{[t/\kappa]\kappa}^t\!\!\|p_{3\kappa}(s-)\|^2ds
			\\
			&\rightarrow 0.
		\end{aligned}
	\end{equation*}
	
Taking limit in \eqref{eq:2.3.7} in weak star sense  as $\kappa\to0$, we obtain
	\begin{equation*}
		\begin{aligned}
			\xi(t)-p(0)+\!\!\int_0^t {\mu}\xi(s)ds-\!\!\int_0^{t}(\mathcal{L}^{\star}-\mathcal{C})\xi(s) ds
			=\!\!\int_0^t\mathcal{B}^{\star}\xi(s)dY_s+\!\!\int_0^{t}\mathcal{C}\xi(s-)dZ_s.
		\end{aligned}
	\end{equation*}
	This is precisely  equation (\ref{eq:2.1.23}). Then the  proof of this lemma follows from Lemma \ref{t1.5.1}.
\end{proof}

{\bf Proof of Theorem \ref{thm:2.3.1}.}
We   integrate (\ref{eq:2.2.8}), (\ref{eq:2.2.7}) and (\ref{eq:2.2.9}) over interval $[t_i, t_{i+1}]$, then take expectation and sum them up to obtain 
\begin{equation*}
	\begin{aligned}
		&\tilde{E}\|p_\kappa^{i+1}\|^2-\tilde{E}\|p_\kappa^i\|^2+\tilde{E}\int_{t_i}^{t_{i+1}}
		\big(\frac{2}{3}\mu(\|p_{1\kappa}(s)\|^2+\|p_{2\kappa}(s)\|^2+\|p_{3\kappa}(s)\|^2)
		\\
		&-\vert\mathcal{B}^{\star}p_{1\kappa}(s)\vert_{D}^2
		-2\langle(\mathcal{L}^{\star}-\mathcal{C})p_{2\kappa}(s),p_{2\kappa}(s) \rangle\big)ds=\tilde{E}\int_{t_i}^{t_{i+1}}(\lambda^2-1)\|p_{3\kappa}(s-)\|^2dZ_s.
	\end{aligned}
\end{equation*}
Summing up this equation in $i$ from $0$ up to $r-1$, we get
\begin{equation}
	\label{eq:2.3.16}
	\begin{aligned}
		\tilde{E}\|p_k^r\|^2-\tilde{E}\|p_\kappa^0\|^2+\tilde{E}\int_0^{t_r} \bigg(\frac{2}{3}\mu(\|p_{1\kappa}(s)\|^2+\|p_{2\kappa}(s)\|^2+\|p_{3\kappa}(s)\|^2)\\
		-\vert\mathcal{B}^{\star}p_{1\kappa}(s)\vert_D^2
		-2\langle(\mathcal{L}^{\star}-\mathcal{C})p_{2\kappa}(s),p_{2\kappa}(s) \rangle\bigg)ds=\tilde{E}\int_0^{t_r}(\lambda^2-1)\|p_{3k}(s-)\|^2dZ_s
	\end{aligned}
\end{equation}

For any $t\in [t_r,t_{r+1}]$, integrating (\ref{eq:2.2.7}) on $[t_r, t]$, we have
\begin{equation}
	\label{eq:2.3.17}
	\tilde{E}\|p_{1\kappa}(t)\|^2-\tilde{E}\|p_\kappa^r\|^2+\tilde{E}\int_{t_r}^t \frac{2}{3}\mu\|p_{1\kappa}(s)\|^2ds=\tilde{E}\int_{t_r}^t\vert\mathcal{B}^{\star}p_{1\kappa}(s)\vert_D^2ds.
\end{equation}

Adding  (\ref{eq:2.3.16}) to (\ref{eq:2.3.17}), we obtain 
\begin{equation}
	\label{eq:2.3.18}
	\begin{aligned}
		&\tilde{E}\|p_{1\kappa}(t)\|^2\!-\!\tilde{E}\|p_0\|^2\!+\!\tilde{E}\!\int_0^t\!\! \frac{2}{3}\mu\|p_{1\kappa}(s)\|^2ds\!-\!2\tilde{E}\!\int_0^{[t/\kappa]\kappa}\!\! \langle(\mathcal{L}^{\star}\!-\!\mathcal{C})p_{2\kappa}(s),p_{2\kappa}(s) \rangle ds\\
		&+\tilde{E}\int_0^{[t/\kappa]\kappa}\frac{2}{3}\mu(\|p_{2\kappa}(s)\|^2+\|p_{3\kappa}(s)\|^2)ds
		\\
		&=\tilde{E}\int_0^t\vert\mathcal{B}^{\star}p_{1\kappa}(s)\vert_{D}^2ds+\tilde{E}\int_0^{[t/\kappa]\kappa}(\lambda^2-1)\|p_{3\kappa}(s-)\|^2dZ_s.
	\end{aligned}
\end{equation}

Define
\begin{equation*}
	\begin{aligned}
		S_k^1:={}&\tilde{E}\|p(t)\|^2-2\tilde{E}\int_0^{[t/\kappa]\kappa}\langle (\mathcal{L}^{\star}-\mathcal{C})p(s),p(s)\rangle ds+\tilde{E}\int_0^t \frac{2}{3}\mu\|p(s)\|^2ds
		\\
		&+\tilde{E}\int_0^{[t/\kappa]\kappa} \frac{4}{3}\mu\|p(s)\|^2ds-\tilde{E}\int_0^t\vert\mathcal{B}^{\star}p(s)\vert_{D}^2 ds-\tilde{E}\int_0^{[t/\kappa]\kappa}(\lambda^2-1)\|p_{s-}\|^2dZ_s,\\ 
		S_\kappa^2:={}&-2\tilde{E}(p(t),p_{1\kappa}(t))+4\tilde{E}\int_0^{[t/\kappa]\kappa} \langle(\mathcal{L}^{\star}-\mathcal{C})p_{2\kappa}(s),p(s)\rangle ds
		\\
		&-\frac{4}{3}\mu \tilde{E}\int_0^t(p(s),p_{1\kappa}(s))ds-\frac{4}{3}\mu \tilde{E}\int_0^{[t/\kappa]\kappa}(p(s),p_{2\kappa}(s)+p_{3\kappa}(s))ds+
		\\
		&2\tilde{E}\int_0^t\int\mathcal{B}^{\star}p(s)D(\mathcal{B}^{\star}p_{2k}(s))^Tdxds+2\tilde{E}\int_0^t(\lambda^2-1)(p_{s-},p_{3k}(s-))dZ_s,
\end{aligned} 
	\end{equation*}
	
\begin{equation*}
	\begin{aligned}	
		S_\kappa^3:={}&\tilde{E}\|p_{1\kappa}(t)\|^2+\tilde{E}\int_0^t \frac{2}{3}\mu\|p_{1\kappa}(s)\|^2ds-2\tilde{E}\int_0^{[t/\kappa]\kappa}\langle(\mathcal{L}^{\star}-\mathcal{C})p_{2\kappa}(s),p_{2k}(s) \rangle ds
		\\
		&+\tilde{E}\int_0^{[t/\kappa]\kappa}\frac{2}{3}\mu(\|p_{2\kappa}(s)\|^2+\|p_{3\kappa}(s)\|^2)ds
		-\tilde{E}\int_0^t\vert\mathcal{B}^{\star}p_{2\kappa}(s)\vert_{D}^2ds
		\\
		&-\tilde{E}\int_0^{[t/\kappa]\kappa}(\lambda^2-1)\|p_{3\kappa}(s-)\|^2dZ_s.
\end{aligned}
\end{equation*}
 We now consider the convergence of these items in $L^2(\Omega;H)$  as $\kappa\rightarrow 0$ 
\begin{equation*}
	\begin{aligned}
		S_{\kappa}^1\rightarrow{}&\tilde{E}\|p(t)\|^2-2\tilde{E}\int_0^t\langle (\mathcal{L}^{\star}-\mathcal{C})p(s),p(s)\rangle
		ds+2\tilde{E}\int_0^t\mu\|p(s)\|^2ds
		\\
		&-\tilde{E}\int_0^t\vert\mathcal{B}^{\star}p(s)\vert_{D}^2ds-\tilde{E}\int_0^t(\lambda^2-1)\|p(s-)\|^2dZ_s
		\\
		={}&\|p_0\|^2,
		\\
		S_{\kappa}^2\rightarrow{}&-2\tilde{E}\|p(t)\|^2+4\tilde{E}\int_0^t\langle (\mathcal{L}^{\star}-\mathcal{C})p(s),p(s)\rangle ds-4\mu \tilde{E}\int_0^t\|p(s)\|^2ds
		\\
		&+2\tilde{E}\int_0^t\vert\mathcal{B}p(s)\vert_{D}^2ds+2\tilde{E}\int_0^t(\lambda^2-1)\|p(s-)\|^2dZ_s
		\\
		={}&-2\|p_0\|^2.
	\end{aligned}
\end{equation*}
Notice that  $\displaystyle\lim_{\kappa\to0}S_\kappa^3=\|p_0\|^2$ also follows from  (\ref{eq:2.3.18}). Therefore we have $S_\kappa :=S_\kappa^1+S_\kappa^2+S_\kappa^3\rightarrow 0$ as $\kappa\to 0$.

Choosing
$
\mu=\max\{\frac{3}{2}M,3\alpha,\frac{3}{2}(\varpi_2^2-1)\}+1
$
and by
the uniform elliptic condition, we have
\begin{equation*}
	-\langle(\mathcal{L}^{\star}-\mathcal{C})(p(t)-p_{2\kappa}(t)),p(t)-p_{2\kappa}(t) \rangle+\alpha\|p(t)-p_{2\kappa}(t)\|^2\geq \beta_2\|p(t)-p_{2\kappa}(t)\|_1^2.
\end{equation*}
Thus
\begin{equation*}
	\begin{aligned}
		S_{\kappa}\geq{}&
		\tilde{E}\|p(t)-p_{1\kappa}(t)\|^2+(\frac{2}{3}\mu-M^2)\tilde{E}\int_0^t\|p(s)-p_{1\kappa}(s)\|^2ds
		\\
		&+(\frac{2}{3}\mu-2\alpha)\tilde{E}\int_0^{[t/\kappa]\kappa}\|p(s)-p_{2\kappa}(s)\|^2ds
		\\
		&+(\frac{2}{3}\mu-\lambda^2+1)\tilde{E}\int_0^{[t/k]k}\|p(s)-p_{3\kappa}(s)\|^2ds
		\geq 0,
	\end{aligned}
\end{equation*}
which implies that for any $\forall t\in[0, T]$, as $\kappa\to0$
\begin{equation*}
	\tilde{E}\|p(t)-p_{1\kappa}(t)\|^2\rightarrow 0,~\tilde{E}\int_0^t\|p(s)-p_{1\kappa}(s)\|^2ds\rightarrow 0,~S_\kappa \rightarrow 0.
\end{equation*}
Hence as $\kappa\to0$
\begin{equation*}
	\begin{aligned}
		&p_{1\kappa}(t)\rightarrow p(t)~\text{in}~L^2(\Omega;H)~\text{uniformly for}~t\in [0,T],
		\\
		&p_{1\kappa}\rightarrow p ~\text{in}~\mathbb{L}^2(0,T;H).
	\end{aligned}
\end{equation*}
Similarly, we obtain the convergence of $p_{2\kappa}$ and $p_{3\kappa}$ as $\kappa \to 0$.
\hspace*{\fill}$\qedsymbol$~~

\vskip1mm
As an application of Theorem \ref{thm:2.3.1},   we immediately obtain a  convergence property  for splitting-up  solution $p_{\kappa}^{r+1}$.

\begin{theorem}
	Assume that {\bf H1}-{\bf H3} hold  and $p_0\in L^4(\Omega;H)$. Then the splitting-up  solution $p_{\kappa}^{r+1}$ converges to the exact solution $p(t_{r+1})$ in $L^2(\Omega;H)$ as $\kappa\rightarrow 0$.
\end{theorem}

For $\phi\in V$ and $\tau \in [0, T]$, define two processes $\psi$ and $\zeta$ by  
\begin{equation*}
	\begin{aligned}
		\psi(\tau)&=(\Gamma_\tau^s(\mathcal{L}^{\star}-\mathcal{C})-(\mathcal{L}^{\star}-\mathcal{C})\Gamma_\tau^s)R_\tau^sQ_\tau^s\phi,
		\\
		\zeta(\tau)&=(\Gamma_\tau^sR_\tau^s\mathcal{B}^{\star}-\mathcal{B}^{\star}\Gamma_\tau^sR_\tau^s)Q_\tau^s\phi.
	\end{aligned}
\end{equation*}
We now estimate the two processes, which  will play an important role in the subsequent analysis.

\begin{lemma}
	\label{lem:1.3.4}
	Assume that {\bf H1-H3} hold and $\phi\in V\cap H^3$. Then for $0\leq s <\tau \leq T$
	\begin{align*}
		&\tilde{E}\|\psi(\tau)\|^2\leq C(\tau-s)\tilde{E}\|\phi\|_3^2 ,\quad\|\tilde{E}(\psi(\tau))\|\leq C(\tau-s)\tilde{E}\|\phi\|_3 ,\\ 
		&\tilde{E}\|\zeta(\tau)\|^2\leq C(\tau-s)\tilde{E}\|\phi\|_3^2,\quad\|\tilde{E}(\zeta(\tau))\|\leq C(\tau-s)\tilde{E}\|\phi\|_3.
	\end{align*}
\end{lemma}

\begin{proof}
	By  Lemma \ref{lem:2.2.5},  we have
	\[
	\|\mathcal{L}^{\star}\phi\| \leq C\|\phi\|_1, \quad \|\mathcal{B}^{\star}\phi\|\leq C \|\phi\|_1.
      \]
	From \eqref{eq:2.2.1-2} it follows that
	\begin{align*}
	\Gamma_{\tau}^s(\mathcal{L}^{\star}-\mathcal{C})R_{\tau}^sQ_{\tau}^s\phi=&(\mathcal{L}^{\star}-\mathcal{C})R_{\tau}^sQ_{\tau}^s\phi-\frac{\mu}{3}\int_s^{\tau}\Gamma_{\tau^{\prime}}^s(\mathcal{L}^{\star}-\mathcal{C})R_{\tau^{\prime}}^sQ_{\tau^{\prime}}^s\phi
           d\tau^{\prime}
          \\
          &+\int_s^{\tau}\mathcal{C}\Gamma_{\tau^{\prime}}^s(\mathcal{L}^{\star}-\mathcal{C})R_{\tau^{\prime}}^sQ_{\tau^{\prime}}^s\phi dZ_{\tau^{\prime}},
		\\
	(\mathcal{L}^{\star}-\mathcal{C})\Gamma_{\tau}^sR_{\tau}^sQ_{\tau}^s\phi=&(\mathcal{L}^{\star}-\mathcal{C})R_{\tau}^sQ_{\tau}^s\phi-\frac{\mu}{3}\int_s^{\tau}(\mathcal{L}^{\star}-\mathcal{C})\Gamma_{\tau^{\prime}}^sR_{\tau^{\prime}}^sQ_{\tau^{\prime}}^s\phi
                                                                                   d\tau^{\prime}
          \\
          &+\int_s^{\tau}(\mathcal{L}^{\star}-\mathcal{C})\mathcal{C}\Gamma_{\tau^{\prime}}^sR_{\tau^{\prime}}^sQ_{\tau^{\prime}}^s\phi dZ_{\tau^{\prime}}.
	\end{align*}
	Let $U_1=\Gamma_{\tau^{\prime}}^s(\mathcal{L}^{\star}-\mathcal{C})-(\mathcal{L}^{\star}-\mathcal{C})\Gamma_{\tau^{\prime}}^s$, $U_2=\mathcal{C}\Gamma_{\tau^{\prime}}^s(\mathcal{L}^{\star}-\mathcal{C})-(\mathcal{L}^{\star}-\mathcal{C})\mathcal{C}\Gamma_{\tau^{\prime}}^s$, then
	\begin{equation}\label{eq:2.3.26-2}
		\psi(\tau)=-\frac{\mu}{3}\int_s^{\tau}U_1R_{\tau^{\prime}}^sQ_{\tau^{\prime}}^s\phi d\tau^{\prime}+\int_s^{\tau}U_2R_{\tau^{\prime}}^sQ_{\tau^{\prime}}^s\phi dZ_{\tau^{\prime}}.
              \end{equation}
	Since $Z_{\tau^{\prime}}-\tau^{\prime}$ is a martingale, we have
	\begin{equation}\label{eq:2.3.26-1}
          \begin{aligned}
          \tilde{E}\|\psi(\tau)\|^2\leq& 2\tilde{E}\left(\int_s^{\tau}\|-\frac{\mu}{3}U_1R_{\tau^{\prime}}^sQ_{\tau^{\prime}}^s\phi\|
            d\tau^{\prime}\right)^2+4\tilde{E}\int_s^{\tau}\|U_2R_{\tau^{\prime}}^sQ_{\tau^{\prime}}^s\phi\|^2d\tau^{\prime}
          \\
          &+4\tilde{E}\left(\int_s^{\tau}\|U_2R_{\tau^{\prime}}^sQ_{\tau^{\prime}}^s\phi\| d\tau^{\prime}\right)^2
          \\
          \leq& C(\tau-s)\tilde{E}\|\phi\|_3^2
          \end{aligned}
        \end{equation}

 From \eqref{eq:2.3.26-2}, we have 
 \begin{align*}
\|\tilde{E}(\psi(\tau))\|\leq &
                                \frac{\mu}{3}\int_s^{\tau}\|\tilde{E}(U_1R_{\tau^{\prime}}^sQ_{\tau^{\prime}}^s\phi)\|d\tau^{\prime}
                                +\int_s^{\tau}\|\tilde{E}(U_2R_{\tau^{\prime}}^sQ_{\tau^{\prime}}^s\phi)\|d\tau^{\prime}
  \\
  \leq &C(\tau-s)\tilde{E}\|\phi\|_3.
\end{align*}

Similarly, we obtain the estimates for $\zeta$.

\end{proof}

\begin{theorem}\label{p1.4.4}
	Assume that {\bf H1}-{\bf H3} hold and  $p(t)\in V\cap H^3$ is the solution of \eqref{eq:2.1.23}. Then $p_{\kappa}^{r+1}$ converges to  $p(t_{r+1})$ as $\kappa\to 0$, and satisfies
	\begin{align*}
		(\tilde{E}\|p(t_{r+1})-p_{\kappa}^{r+1}\|^2)^{1/2} &\leq C\kappa^{1/2},
		\\
		\|\tilde{E}(p(t_{r+1})-p_{\kappa}^{r+1})\| & \leq C\kappa.
	\end{align*}
\end{theorem}

\begin{proof}
	Let $\varepsilon(t)=\Gamma_t^sR_t^sQ_t^s\phi$ for $\phi\in C_0^{\infty}$ and $0\leq s < t\leq T$.  Then
	\begin{align*}
		d\varepsilon(t)={}&\Gamma_t^s(-\frac{\mu}{3}R_t^sQ_t^s\phi dt+(\mathcal{L}^{\star}-\mathcal{C})R_t^sQ_t^s\phi dt+R_t^s(-\frac{\mu}{3}Q_t^s\phi dt+\mathcal{B}^{\star}Q_t^s\phi dY_t))
		\\
		&-\frac{\mu}{3}\varepsilon(t)dt+\mathcal{C}\varepsilon(t-)dZ_t
		\\
		={}&-\mu\varepsilon(t)dt+\Gamma_t^s(\mathcal{L}^{\star}-\mathcal{C})R_t^sQ_t^s\phi dt+\Gamma_t^sR_t^s\mathcal{B}^{\star}Q_t^s\phi dY_t+\mathcal{C}\varepsilon(t-)dZ_t.
		\\
		={}&-\mu\varepsilon(t)dt+(\mathcal{L}^{\star}-\mathcal{C})\varepsilon(t)dt+\mathcal{B}^{\star}\varepsilon(t)dY_t+\mathcal{C}\varepsilon(t-)dZ_t
		\\
		&+\psi(t) dt+\zeta(t) dY_t.
	\end{align*}
	Let $\gamma(t)=\varepsilon(t)-p(t)$. By It\^o's formula, we have
	\begin{align*}
		\tilde{E}\|\gamma(t)\|^2 &= \tilde{E}\|\gamma(s)\|^2-2\mu\tilde{E}\int_s^t\|\gamma(\tau)\|^2d\tau+2\tilde{E}\int_s^t\langle (\mathcal{L}^{\star}-\mathcal{C})\gamma(\tau), \gamma(\tau)\rangle d\tau
		\\
		&+2\tilde{E}\int_s^t(\gamma(\tau),\psi(\tau))d\tau + \int_s^t\|\mathcal{B}^{\star}\gamma(\tau)+\zeta(\tau)\|^2d\tau+\int_s^t(\lambda^2-1)\|\gamma(\tau)\|^2d\tau.
	\end{align*}
	From Lemma \ref{lem:2.2.5} and inequality \eqref{eq:2.1.22}, we have 
	\begin{equation*}
		\tilde{E}\|\gamma(t)\|^2\leq \tilde{E}\|\gamma(s)\|^2+C\int_s^t\tilde{E}\|\gamma(\tau)\|^2d\tau +\int_s^t\tilde{E}\|\psi(\tau)\|^2d\tau + 2\int_s^t\tilde{E}\|\zeta(\tau)\|^2d\tau.
	\end{equation*}
	Applying  the Gronwall lemma and Lemma \ref{lem:1.3.4} gives
	\begin{align*}
		\tilde{E}\|\gamma(t)\|^2
		&\leq \left(\tilde{E}\|\gamma(s)\|^2+\int_s^t\tilde{E}\|\psi(\tau)\|^2d\tau +2\int_s^t\tilde{E}\|\zeta(\tau)\|^2d\tau \right) e^{C(t-s)}
		\\
		& \leq (\tilde{E}\|\gamma(s)\|^2+C(t-s)^2\tilde{E}\|\phi\|_3^2)e^{C(t-s)}.
	\end{align*}
	Taking $s=t_i$, $t=t_{i+1}$ and $\phi=p_\kappa^i$, we obtain
	\begin{equation}\label{eq:2.3.26}
		\tilde{E}\|p_{\kappa}^{i+1}-p(t_{i+1})\|^2\leq (\tilde{E}\|p_{\kappa}^i-p(t_i)\|^2+C\kappa^2\tilde{E}\|p_{\kappa}^i\|_3^2)e^{C\kappa}.
	\end{equation}
Iterating the above equation in $i$ from $i=0$ to $r$ and applying Theorem \ref{thm:2.2.1}, we have
	\begin{align*}
          &\tilde{E}\|p_{\kappa}^{r+1}-p(t_{r+1})\|^2
          \\
          \leq &\left[(\tilde{E}\|p_{\kappa}^{r-1}-p(t_{r-1})\|^2+C\kappa^2\tilde{E}\|p_{\kappa}^{r-1}\|_3^2)e^{C\kappa}+C\kappa^2\tilde{E}\|p_{\kappa}^r\|_3^2\right]e^{C\kappa}
          \\
          \leq &\cdots \leq C\kappa^2\sum\limits_{i=0}^{r}\tilde{E}\|p_{\kappa}^i\|_3^2e^{CT}
		\leq C\kappa.
	\end{align*}
        
Noticing
\begin{equation*}
\tilde{E}(\gamma(t))=\tilde{E}(\gamma(s))-\mu\int_s^t\tilde{E}(\gamma(\tau))d\tau+\int_s^t\tilde{E}(\mathcal{L}^{\star}\gamma(\tau))d\tau+\int_s^t\tilde{E}(\psi(\tau))d\tau,
\end{equation*}
we have
\begin{equation*}
\|\tilde{E}(\gamma(t))\|\leq \|\tilde{E}(\gamma(s))\|+C\int_s^t\|\tilde{E}(\gamma(\tau))\|d\tau+\int_s^t\|\tilde{E}(\psi(\tau))\|d\tau.
\end{equation*}
Then by the Gronwall lemma, from Lemma \ref{lem:1.3.4}, we have
\begin{align*}
  \|\tilde{E}(\gamma(t))\|&\leq (\|\tilde{E}(\gamma(s))\|+\int_s^t\|\tilde{E}(\psi(\tau))\|d\tau)e^{C(t-s)}
  \\
  &\leq (\|\tilde{E}(\gamma(s))\| + C(t-s)^2\tilde{E}\|\phi\|_3)e^{C(t-s)}.
\end{align*}
Taking $s=t_i, t=t_{i+1}$ and $\phi=p_{\kappa}^i$, we obtain
\begin{equation}\label{eq:4.7}
\|\tilde{E}(p_{\kappa}^{i+1}-p(t_{i+1}))\|\leq (\|\tilde{E}(p_{\kappa}^i-p(t_i))\|+C\kappa^2\tilde{E}\|p_{\kappa}^i\|_3)e^{C\kappa}.
\end{equation}
Integrating the above equation in $i$ from $i=0$ to $r$ and applying Theorem \ref{thm:2.2.1}, we have
\begin{align*}
  &\|\tilde{E}(p_{\kappa}^{r+1}-p(t_{r+1}))\|
  \\
  \leq &\left[ (\|\tilde{E}(p_{\kappa}^{r-1}-p(t_{r-1}))\|+C\kappa^2\tilde{E}\|p_{\kappa}^{r-1}\|_3)e^{C\kappa}+C\kappa^2\tilde{E}\|p_{\kappa}^r\|_3\right]e^{C\kappa}
  \\
  \leq & \cdots \leq C\kappa^2\sum\limits_{i=0}^{r}\tilde{E}\|p_{\kappa}^i\|_3e^{CT}\leq C\kappa.
\end{align*}

The proof is complete.
\end{proof}

{\bf Remark} We note that $p_{1\kappa}$, $p_{2\kappa}$ and $p_{3\kappa}$ are defined by the continuous solution operators $Q_t^s$, $R_t^s$ and $\Gamma_t^s$ for
\eqref{eq:2.2.1}, \eqref{eq:2.2.1-1} and \eqref{eq:2.2.1-2}, respectively. They are splitting up solutions for continuous problems, not numerical ones.
In the next section, we will consider temporal discretizations of \eqref{eq:2.2.1}, \eqref{eq:2.2.1-1}, \eqref{eq:2.2.1-2} and construct semi-discretized splitting-up approximations for the exact solution $p$. 

\section{Semi-discretization and error analysis}\label{sec4}
In this section, we construct  a  semi-discretized splitting-up  scheme by discretizing   \eqref{eq:2.2.1}-\eqref{eq:2.2.1-2} with the finite difference method and  investigate its error estimate.

On each interval $[t_r,t_{r+1}]$  $(r=0, 1, 2, \cdots, N-1)$, we apply the Euler-Maruyama scheme to \eqref{eq:2.2.1},  backward implicit Euler method to  (\ref{eq:2.2.1-1}) and forward explicit Euler method to (\ref{eq:2.2.1-2}) to obtain a semi-discrete scheme:
\begin{align}\label{alg:2.6.1}
	p_{1\kappa,r+1}-p_{1\kappa,r}
	&=-\frac{\mu}{3}p_{1\kappa,r}\kappa+\mathcal{B}^{\star}p_{1\kappa,r}(Y_{t_{r+1}}-Y_{t_r}),
	\\
	\label{alg:2.6.2}
	p_{2\kappa,r+1}-p_{2\kappa,r}
	&=((\mathcal{L}^{\star}-\mathcal{C})p_{2\kappa,r+1}-\frac{\mu}{3}p_{2\kappa,r+1})\kappa,
	\\
	\label{alg:2.6.3}
	p_{3\kappa,r+1}-p_{3\kappa,r}
	&=-\frac{\mu}{3}p_{3\kappa,r}\kappa+\mathcal{C}p_{3\kappa,r}(Z_{t_{r+1}}-Z_{t_r}), 
\end{align}
where $Z_{t_{r+1}}-Z_{t_r}$ is the number of  jumps of Poisson process $Z_t$ within  time interval $[t_r, t_{r+1}]$. The iterative solutions  $p_{i\kappa,r}$ of equations \eqref{alg:2.6.1}-\eqref{alg:2.6.3} are numerical  approximations to  $p_{i\kappa}(t_r)$ for $i=1,2,3$. And each jump quantity is approximated by $\mathcal{C}p_{3\kappa,r}=(\lambda-1)p_{3\kappa,r}$.

Let $\bar{Q}_{t_{r+1}}^{t_r}$, $\bar{R}_{t_{r+1}}^{t_r}$ and  $\bar{\Gamma}_{t_{r+1}}^{t_r}$ successively denote the solution operators of above equations. In terms of these settings, we define  a
discrete splitting-up approximate solution $p_{\kappa,r}$ $(r=0, 1, 2, \cdots, N)$ of Zakai equation \eqref{eq:2.1.20} as 
\begin{equation}\label{alg:2.6.4}
	p_{\kappa,r+1}=\bar{\Gamma}_{t_{r+1}}^{t_r} \bar{R}_{t_{r+1}}^{t_r}\bar{Q}_{t_{r+1}}^{t_r}p_{\kappa, r},
	\quad p_{\kappa,0}=p_0.
\end{equation}

Now, we are ready to state and prove the main result of this work.

\begin{theorem}
	\label{thm:2.6.1-1}
	Assume {\bf H1}-{\bf H3}. Then the discrete splitting-up solution  $p_{\kappa,r+1}$ converges to the exact solution $p(t_{r+1})$  as $\kappa \to 0$ and satisfies
	\begin{equation}\label{eq:2.6-1}
		\{\tilde{E}\|p(t_{r+1})-p_{\kappa,r+1}\|^2\}^{1/2}\leq C\sqrt{\kappa}.
	\end{equation}
\end{theorem}

\begin{proof}
	By  the Milstein Theorem in \cite[Theorem 1.1, page 12]{MiGN}, we only need to show that 
	\begin{align*}
		&\{\tilde{E}\|p(t_{r+1})-p_{\kappa,r+1}\|^2\}^{1/2}\leq C\|p(t_r)\|\kappa,
		\\
		&\|\tilde{E}(p(t_{r+1})-\bar{\Gamma}_{t_{r+1}}^{t_r}\bar{R}_{t_{r+1}}^{t_r}\bar{Q}_{t_{r+1}}^{t_r}p(t_r))\|\leq C\|p(t_r)\|\kappa^2.	
	\end{align*}
First, we estimate the truncation error in the mean square sense.  Assume   $p(t_r)=p_{\kappa,r}$, then 
	\begin{align*}
		&\tilde E \|p(t_{r+1})-P_{\kappa,r+1}\|^2=
		\tilde{E}\|p(t_{r+1})-\bar{\Gamma}_{t_{r+1}}^{t_r} \bar{R}_{t_{r+1}}^{t_r}\bar{Q}_{t_{r+1}}^{t_r}p(t_r)\|^2
		\\
		&\leq 3\tilde{E}\|p(t_{r+1})-\Gamma_{t_{r+1}}^{t_r} R_{t_{r+1}}^{t_r}Q_{t_{r+1}}^{t_r}p(t_r)\|^2
		\\
		&~~+3\tilde{E}\|\Gamma_{t_{r+1}}^{t_r}
		R_{t_{r+1}}^{t_r}Q_{t_{r+1}}^{t_r}p(t_r)-\Gamma_{t_{r+1}}^{t_r}R_{t_{r+1}}^{t_r}\bar{Q}_{t_{r+1}}^{t_r}p(t_r)\|^2
		\\
		&~~+3\tilde{E}\|\Gamma_{t_{r+1}}^{t_r}R_{t_{r+1}}^{t_r}\bar{Q}_{t_{r+1}}^{t_r}p(t_r)-\Gamma_{t_{r+1}}^{t_r}\bar{R}_{t_{r+1}}^{t_r}\bar{Q}_{t_{r+1}}^{t_r}p(t_r)\|^2
		\\
		&~~+3\tilde{E}\|\Gamma_{t_{r+1}}^{t_r}\bar{R}_{t_{r+1}}^{t_r}\bar{Q}_{t_{r+1}}^{t_r}p(t_r)-\bar{\Gamma}_{t_{r+1}}^{t_r}\bar{R}_{t_{r+1}}^{t_r}\bar{Q}_{t_{r+1}}^{t_r}p(t_r)\|^2
		\\
		&:=S_0+S_1+S_2+S_3.
	\end{align*}
	Inequality  \eqref{eq:2.3.26} directly implies that 
	$S_0\leq C\kappa^2\|p(t_r)\|^2$.
	From \cite[Theorem 4.6]{GLMV2011} and \cite[Theorem 67]{Protter2005} it follows that
	\begin{equation}\label{eq:2.6.1-1}
		\tilde{E}\|\Gamma_{t_{r+1}}^{t_r}\phi\|^2\leq C\tilde{E}\|\phi\|^2,\quad \tilde{E}\|R_{t_{r+1}}^{t_r}\phi\|^2\leq C\tilde{E}\|\phi\|^2,\quad \forall\phi\in L^2(\Omega;H),
	\end{equation}
	which leads to
	\begin{equation*}
		S_1\leq C\tilde{E} \|Q_{t_{r+1}}^{t_r}p(t_r)-\bar{Q}_{t_{r+1}}^{t_r}p(t_r)\|^2\leq C\|p(t_r)\|^2\kappa^2.
	\end{equation*}
	Here we have used the convergence theorem of the Euler-Maruyama method  \cite{KaSh1995}.
	Similarly, we obtain
	$S_3\leq C\|p(t_r)\|^2\kappa^2$.
	
	According to the convergence theorem  of the backward Euler scheme  \cite[page 343]{KlPE2011}, we have
	\begin{equation*}
		S_2\leq C\tilde{E}\|[R_{t_{r+1}}^{t_r}-\bar{R}_{t_{r+1}}^{t_r}]\bar{Q}_{t_{r+1}}^{t_r}p(t_r)\|^2\leq C\|\bar Q^{t_r}_{t_{r+1}}p(t_r)\|^2\kappa^4\leq C\|p(t_r)\|^2\kappa^4.
	\end{equation*}
Therefore, we have proved that 
	\begin{equation}\label{eq:2.6.2-1}
		\{\tilde{E}\|p(t_{r+1})-p_{\kappa,r+1}\|^2\}^{1/2}\leq C\|p(t_r)\|\kappa.
	\end{equation}

	Next, we estimate the truncation error in expectation. Notice that
	\begin{align*}
		&\|\tilde E(p(t_{r+1})-p_{\kappa,r+1})\|=
		\|\tilde{E}(p(t_{r+1})-\bar{\Gamma}_{t_{r+1}}^{t_r}\bar{R}_{t_{r+1}}^{t_r}
		\bar{Q}_{t_{r+1}}^{t_r}p(t_r))\|
		\\
		&\leq \|\tilde{E}(p(t_{r+1})-\Gamma_{t_{r+1}}^{t_r} R_{t_{r+1}}^{t_r}Q_{t_{r+1}}^{t_r}p(t_r))\| \\
		&~~+\|\tilde{E}(\Gamma_{t_{r+1}}^{t_r}
		R_{t_{r+1}}^{t_r}Q_{t_{r+1}}^{t_r}p(t_r)-\Gamma_{t_{r+1}}^{t_r}R_{t_{r+1}}^{t_r}\bar{Q}_{t_{r+1}}^{t_r}p(t_r))\|
		\\
		&~~+\|\tilde{E}(\Gamma_{t_{r+1}}^{t_r}R_{t_{r+1}}^{t_r}\bar{Q}_{t_{r+1}}^{t_r}p(t_r)-\Gamma_{t_{r+1}}^{t_r}\bar{R}_{t_{r+1}}^{t_r}\bar{Q}_{t_{r+1}}^{t_r}p(t_r))\|
		\\
		&~~+\|\tilde{E}(\Gamma_{t_{r+1}}^{t_r}\bar{R}_{t_{r+1}}^{t_r}\bar{Q}_{t_{r+1}}^{t_r}p(t_r)-\bar{\Gamma}_{t_{r+1}}^{t_r}\bar{R}_{t_{r+1}}^{t_r}\bar{Q}_{t_{r+1}}^{t_r}p(t_r))\|
		\\
		:=& S_0^{\prime}+S_1^{\prime} +S_2^{\prime}+S_3^{\prime}.
	\end{align*}
	From \eqref{eq:4.7} it follows that $S_0^{\prime}\leq C\|p(t_r)\|\kappa^2$.
	By \eqref{eq:2.6.1-1} and the convergence theory of Euler-Maruyama method in \cite{KaSh1995}, we have
	\begin{equation*}
		S_1^{\prime}\leq C\|\tilde{E}[Q_{t_{r+1}}^{t_r}-\bar{Q}_{t_{r+1}}^{t_r}]p(t_r)\|\leq C \|p(t_r)\|\kappa^2.
	\end{equation*}
	Similarly, we obtain
	$S_3^{\prime}\leq C\|p(t_r)\|\kappa^2$.

	From \cite[page 343]{KlPE2011}  and the convergence theory of implicit Euler scheme, we get
	\begin{align*}
		S_2^{\prime}&\leq C\|\tilde{E} [R_{t_{r+1}}^{t_{r}}-\bar{R}_{t_{r+1}}^{t_r}]
		\bar{Q}_{t_{r+1}}^{t_r}p(t_r)\|\leq C\tilde{E}\|\bar{Q}_{t_{r+1}}^{t_r}p(t_r)\|\kappa^2
		\leq C\|p(t_r)\|\kappa^2.
	\end{align*}
Summing the above estimates, we obtain 
	\begin{equation}\label{eq:2.6.5}
		\|\tilde{E}(p(t_{r+1})-\bar{\Gamma}_{t_{r+1}}^{t_r}\bar{R}_{t_{r+1}}^{t_r}\bar{Q}_{t_{r+1}}^{t_r}p(t_r))\|\leq C\|p(t_r)\|\kappa^2.
	\end{equation}
	By the estimates in \eqref{eq:2.6.2-1} and \eqref{eq:2.6.5},  and  Milstein's  theorem  \cite{MiGN}, we obtain \eqref{eq:2.6-1}. Therefore, the  proof is completed.
\end{proof}

\section{Numerical experiments}\label{sec5}
In this section, we apply our algorithm to a linear filtering model and a nonlinear filtering model to illustrate our theoretical results on error estimates.

We use  a spectral Galerkin method to discretize the spatial variable   with an  $n$-dimensional    subspace whose  basis functions
are given by 
\begin{equation*}
	e_{i}(x)=\sqrt{\frac{\phi(x)}{(i-1)!}}f_{i-1}(x),\quad x\in\mathbb{R},\quad i=1,2,\cdots n,
\end{equation*}
where $f_{i-1}(x)$ is a  Hermite polynomial of order $i-1$ and $\phi(x)=(2\pi)^{-1/2}e^{-x^{2}/2}$. Due to \cite{HeJa2013}, the spatial Galerkin discretization errors are expected to decrease   exponentially with respect to the dimension $n$. 
To proceed with numerical experiments, we need to simulate the sample trajectories of the Poisson process with density function $\lambda(X_t)$. For convenience, we assume the jump time  $\tau_i$
may occur only at $t_r$. To determine  the  jump time, we first calculate  a discrete sample trajectory $X_{t_r}$ $(r=0,1,\cdots, N)$ in terms of  equation \eqref{eq:2.1.1} and compute the density
function  $\lambda(X_{t_r})$. For any $i\leq r$, define  $\mathcal{T}_i(t_r):=\kappa \sum\limits_{j=i}^{r-1}\lambda(X_{t_j})$. Assuming $\tau_0=0$ and $\tau_{i}$ is the present jump time; we will describe a
criterion  for finding the next jump time $\tau_{i+1}$. Let  $\mathcal{E}$ be a random number  generated by  a  random variable with unit exponential distribution and  independent of $X_{t_r}$. Then the next jump time  $\tau_{i+1}$ is defined as
the first time that $\mathcal{T}_{i}(t_r)$ exceeds  $\mathcal{E}$.

{\bf Example 1.}
Consider a linear filtering model 
\begin{align*}
	&dX_t=0.5X_tdt+2 dw_t,\quad X_0\sim \mathcal{N}(5,0.01),
	\\
	&dY_t=X_tdt+0.5dw_t+dv_t,\quad Y_0=0,
	\\
	&Z_t~ \text{is a doubly  stochastic Poisson process with the intensity of} 
	~\lambda X_t^2.
\end{align*}
The corresponding Zakai equation reads
\begin{equation*}
	\begin{aligned}
		dp(t)=&\big(2 \frac{\partial^2p(t)}{\partial x^2}- \frac{1}{2} \frac{\partial p(t)}{\partial x}-\lambda x^2p(t)+p(t)\big)dt
		\\
		&+\big(-\frac{2\sqrt{5}}{5}\frac{\partial p(t)}{\partial x}
		+\frac{11}{5}\sqrt{5}xp(t)\big)dY_t
		+(\lambda x^2-1)p(t-)dZ_t.
	\end{aligned}
\end{equation*}

Taking a time stepsize  $\kappa=0.5\times 10^{-5}$ and  choosing $\lambda=3$, we trace a sample trajectory using the Zakai filter and depict it  in Fig.\ \ref{fig:1.6.1-1}.

\begin{figure}[H]
	\centering
	\includegraphics[width=8cm]{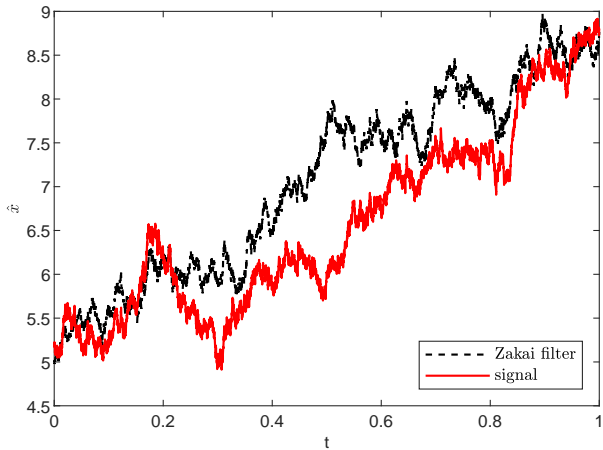}
	\caption{The Zakai filter and signal for $\lambda=3$.}
	\label{fig:1.6.1-1}
\end{figure}


In addition, we   simulate a sample path  for only continuous
observation $Y_t$ with $\lambda =0$  and  mixed observations of $Y_t$ and $Z_t$ with $\lambda=3$.  We trace their corresponding conditional standard deviation versus time $t$, see Fig.\ \ref{fig:1.6.2}.
 It shows that including  information on point process observation   reduces the conditional standard deviation.
\begin{figure}[H]
	\begin{center}
	\includegraphics[width=8cm]{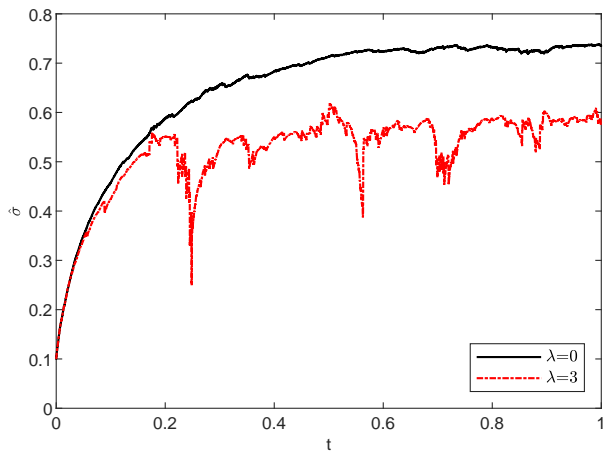}
	\caption{The conditional standard deviation}
	\label{fig:1.6.2}
    \end{center}
\end{figure}

We fix $T=0.25$, $\lambda=2$ and  a stepsize $\kappa=N^{-1}=2^{-20}$  and  then compute $m=500$ reference 'exact' sample
paths. Simultaneously we compute $m$ Zakai filter approximate solutions  for each stepsize $\kappa_i=N_i^{-1}=2^{-i}$ with $i=16, 17, 18, 19$.  
Define  an  error function
\begin{align*}
	d(\kappa_i)=\{\frac{1}{N_im}\sum_{r=1}^{N_i}\sum\limits_{j=1}^{m}
	\vert\hat{X}^{j}(t_r)-\hat{X}_r^j\vert^2\}^{1/2}.
\end{align*}
The dynamic behavior of the  errors as varying   stepsize $\kappa_i$ is exhibited in  Fig.\ \ref{fig:1.6.3}, demonstrating the half order convergence rate.
\begin{figure}[H]
	\centering
	\includegraphics[width=8cm]{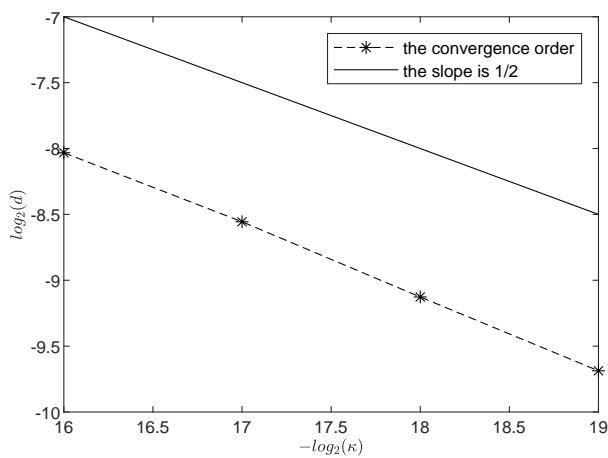}
	\caption{The convergence order of the splitting-up method.}
	\label{fig:1.6.3}
\end{figure}

{\bf Example 2.}  Consider a nonlinear filtering model 
\begin{align*}
	&dX_t=\sin(X_t)dt+2 dw_t,\quad X_0\sim \mathcal{N}(5,0.01),
	\\
	&dY_t=5.5X_tdt+0.5dw_t+dv_t,\quad Y_0=0,
	\\
	&Z_t~ \text{is a doubly  stochastic Poisson process with the intensity of} 
	~\lambda X_t^2.
\end{align*}
The corresponding Zakai equation is
\begin{align*}
	dp(t)=&(2 \frac{\partial^2 p(t)}{\partial x^2}-\sin(x) \frac{\partial p(t)}{\partial
		x}-\cos(x)p(t)-(\lambda x^2-1)p(t))dt\\
	&+(-\frac{2}{5}\sqrt{5}\frac{\partial p(t)}{\partial x}+\frac{11}{5}\sqrt{5}xp(t))dY_t+(\lambda x^2-1)p(t-)dZ_t.
\end{align*}

Set  $\lambda = 3$, $T =0.5$, $\kappa = 0.5\times 10^{-5}$. We  trace a sample signal trajectory using our Zakai filter and depict the approximations in  Fig.\
\ref{fig:1.6.4}.

\begin{figure}[H]
  \centering
\includegraphics[width=8cm]{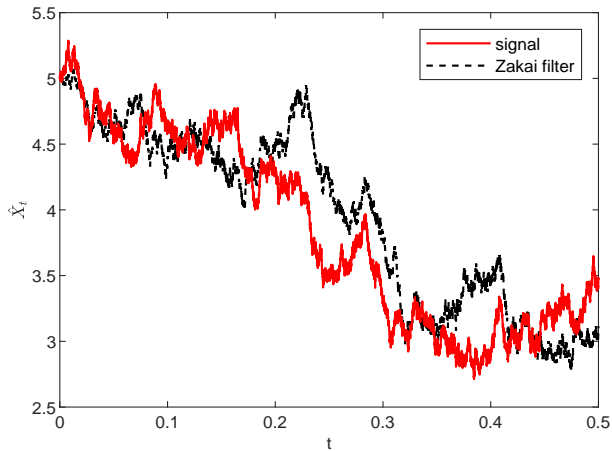}
\caption{The Zakai filter and signal process.}
\label{fig:1.6.4}
\end{figure}

Next, we verify the convergence order in a temporal variable. We compute $m=500$ reference ``exact" Zakai filter solutions by fixing  the stepsize $\kappa = 2^{-20}$. Then we calculate $m$ numerical 
Zakai filter solutions for each stepsize $\kappa_i=2^{-i}$, $i=14, 15, 16, 17$. We  plot the errors in $\log$-$\log$ scale, cf.\ Fig.\ \ref{fig:1.6.6}, which shows that  the convergence order  is of $\frac{1}{2}$. 
\begin{figure}[H]
\centering
\includegraphics[width=8cm]{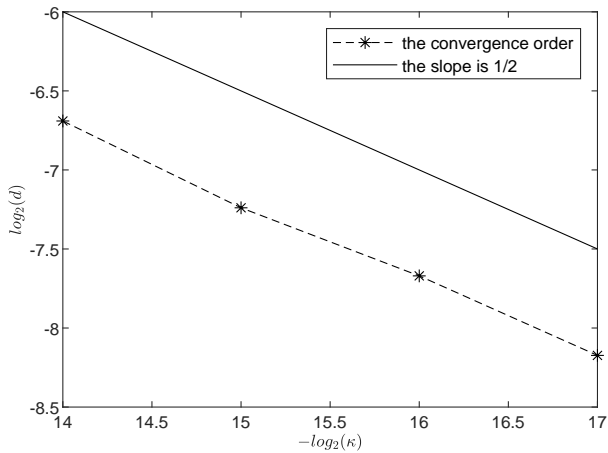}
\caption{The convergence order of time discretized method.}
\label{fig:1.6.6}
\end{figure}

\section{Conclusion and discussion}

In this paper, we considered a nonlinear filtering model  with observations involving a mixture of a Wiener process and a point process. After deriving the corresponding Zakai equation, we constructed the splitting-up scheme where the Zakai equation is decomposed into three equations: A deterministic PDE, an SDE driven by a Wiener process, and an SDE driven by a point process. Then we discretized these equations in the temporal direction by finite difference methods. 
By estimating the errors of these splitting up equations and the errors of the temporal discretization, we derived the half-order convergence result for the proposed numerical scheme using Milstein's fundamental theorem on numerical methods for SDEs. Our current work focuses on  semi-discretizations in time. Future research on this topic includes the construction and error estimates for fully discretized numerical schemes.

\bmhead{Acknowledgments}
The research was partly supported by the National Key R\&D Program (2020YFA0714101, 2020YFA0713601), NSFC (12171199, 11971198), Jilin Provincial Department of science and technology (20210201015GX).

\bibliography{ref}


\end{document}